\documentclass[12pt, usenames,dvipsnames]{amsart}
\usepackage{amssymb,amsmath,amsfonts,amscd,euscript,mathrsfs,bm}
\usepackage{fullpage}
\usepackage[backref=page]{hyperref}

\DeclareMathAlphabet{\mathpzc}{OT1}{pzc}{m}{it}
\usepackage[sans]{dsfont}

\usepackage{tikz}
\usepackage{tikz-cd}
\usetikzlibrary{babel}
\usetikzlibrary{matrix}
\usetikzlibrary{shapes}
\usetikzlibrary{arrows}
\usetikzlibrary{calc,3d}
\usetikzlibrary{decorations,decorations.pathmorphing}
\usetikzlibrary{through}
\tikzset{ext/.style={circle, draw,inner sep=1pt},int/.style={circle,draw,fill,inner sep=1pt},nil/.style={inner sep=1pt}}
\tikzset{exte/.style={circle, draw,inner sep=3pt},inte/.style={circle,draw,fill,inner sep=3pt}}
\tikzset{diagram/.style={matrix of math nodes, row sep=3em, column sep=2.5em, text height=1.5ex, text depth=0.25ex}}
\tikzset{diagram2/.style={matrix of math nodes, row sep=0.5em, column sep=0.5em, text height=1.5ex, text depth=0.25ex}}
\tikzset{every picture/.append style={baseline=-.65ex}}

\usepackage{todonotes}

\usepackage{hyperref}

\usepackage{color}

\newcommand{\nc}{\newcommand}
\newtheorem{theoremA}{Theorem}

\numberwithin{equation}{section}

\newtheorem*{conj*}{Conjecture}
\newtheorem{thm}[equation]{Theorem}
\newtheorem{prop}[equation]{Proposition}
\newtheorem{lem}[equation]{Lemma}
\newtheorem{cor}[equation]{Corollary}
\newtheorem*{cor*}{Corollary}
\newtheorem{rem}[equation]{Remark}
\newtheorem{prop::def}[equation]{Proposition-Definition}

\newtheorem{dfn}[equation]{Definition}

\newcommand{\hT}{\mathpzc{T}}

\newcommand{\hX}{\mathpzc{X}}
\newcommand{\hY}{\mathpzc{Y}}
\newcommand{\hH}{\mathpzc{H}}

\nc{\fB}{\mathfrak{B}}
\nc{\gl}{\mathfrak{gl}}
\nc{\GL}{\mathfrak{GL}}
\nc{\g}{\mathfrak{g}}
\nc{\gh}{\widehat\g}
\nc{\h}{\mathfrak{h}}
\nc{\wfh}{\widehat{\mathfrak{h}}}
\nc{\la}{\lambda}
\nc{\al}{\alpha }
\nc{\be}{\beta }
\nc{\ve}{\varepsilon }
\nc{\om}{\omega }
\nc{\lr}{\text{-}}
\nc{\ta}{\theta}
\nc{\ch}{{\mathop {\rm ch}}}
\nc{\Tr}{{\mathop {\rm Tr}\,}}
\nc{\tr}{{\mathrm tr}}
\nc{\Id}{{\mathop {\rm Id}}}
\nc{\ad}{{\mathop {\rm ad}}}
\nc{\End}{{\mathop {\rm End}}}

\nc{\bra}{\langle}
\nc{\ket}{\rangle}
\nc{\bi}{{\bf i}}
\nc{\pa}{\partial}
\nc{\ld}{\ldots}
\nc{\cd}{\cdots}
\nc{\hk}{\hookrightarrow}
\nc{\T}{\otimes}
\nc{\gr}{\mathrm{gr}}
\nc{\ov}{\overline}

\nc{\msl}{\mathfrak{sl}}
\nc{\mgl}{\mathfrak{gl}}
\nc{\U}{\mathrm U}
\nc{\V}{\EuScript V}
\nc{\cO}{\mathcal{O}}
\nc{\cL}{\mathcal{L}}
\nc{\Res}{{\mathbf{Res}}}
\nc{\Ind}{{\mathbf{Ind}}}
\nc{\Coind}{{\mathbf{CoInd}}}
\nc{\LInd}{{\mathbf{LInd}}}
\nc{\RCoind}{{\mathbf{RCoInd}}}

\nc{\sL}{\mathbf{L}}
\nc{\sR}{\mathbf{R}}
\nc{\sZ}{\mathbf{Z}}

\newcommand{\DAHA}{\mathcal{H}\kern -.4em\mathcal{H}}
\newcommand{\cI}{{\mathcal{I}}}

\newcommand{\cP}{{\mathcal{P}}}

\newcommand{\cD}{{\mathcal D}}
\newcommand{\bD}{{\mathbb D}}  

\newcommand{\bZ}{{\mathbb Z}}

\newcommand{\bI}{{\mathbb I}}
\newcommand{\bP}{{\mathbb P}}

\newcommand{\fp}{{\mathfrak p}}
\newcommand{\fh}{{\mathfrak h}}

\newcommand{\fg}{{\mathfrak g}}
\newcommand{\fa}{{\mathfrak a}}

\newcommand{\fb}{{\mathfrak b}}

\newcommand{\fl}{{\mathfrak l}}

\newcommand{\Hom}{\mathrm{Hom}}
\newcommand{\RHom}{\mathrm{RHom}}

\nc{\bfI}{\mathbf I}
\newcommand{\fC}{{\mathfrak C}}
\nc{\Q}{\mathfrak Q}
\nc{\fr}{\mathfrak r}
\nc{\W}{\mathbb W}
\nc{\bU}{\mathbb U}

\nc{\Gm}{\mathbb{G}_{m}}
\nc{\bA}{\mathbb A}

\newcommand{\cA}{\mathcal{A}}

\DeclareMathOperator{\Ext}{Ext}

\newcommand{\udot}{{\:\raisebox{4pt}{\selectfont\text{\circle*{1.5}}}}}

\newcommand\ttt{\text{-}}

\begin{document}
	
	\title[]
	{Parahoric Lie algebras and \\ parasymmetric Macdonald polynomials}

	\author[Feigin]{Evgeny Feigin}
	\address{Evgeny Feigin:\newline
		School of Mathematical Sciences, Tel Aviv University, Tel Aviv
		69978, Israel
	}
	\email{evgfeig@gmail.com}

	\author[Khoroshkin]{Anton Khoroshkin}
	\address{Anton Khoroshkin: \newline
		Department of Mathematics, University of Haifa, Mount Carmel, 3498838, Haifa, Israel
	}
	
	\email{khoroshkin@gmail.com}
	
	\author[Makedonskyi]{Ievgen Makedonskyi}
	\address{Ievgen Makedonskyi:\newline
		Yanqi Lake Beijing Institute of Mathematical Sciences And Applications (BIMSA), 
		No. 544, Hefangkou Village, Huaibei Town, Huairou District, Beijing 101408.}
	\email{makedonskii\_e@mail.ru}

	\begin{abstract}
		The main goal of this paper is to categorify the specialized parasymmetric (intermediate) Macdonald polynomials. 
		These polynomials depend on a parabolic subalgebra of a simple Lie algebra and generalize the symmetric and nonsymmetric Macdonald polynomials. 
		
		To achieve this we introduce cyclic modules of the parahoric subalgebras of the affine Kac-Moody Lie algebras such that their characters coincide with the specializations of the parasymmetric polynomials at zero and infinity. 
		These cyclic modules are proved to coincide with standard and costandard objects in certain categories of representations of parahoric algebras. We show that
		the categories in question are stratified, i.e. they are graded highest weight categories. As a consequence, we derive
		an analog of the Peter-Weyl theorem describing the bi-module of functions on the parahoric and parabolic Lie groups
		via the mentioned above standard and costandard modules. 
	\end{abstract}
	
	\maketitle

	\setcounter{tocdepth}{1}
	\tableofcontents

	\setcounter{section}{-1}
	\section{Introduction}
	Let $\fg$ be a simple Lie algebra with a finite root system $\Phi$ and a weight lattice $P$. 
	Symmetric Macdonald polynomials $P_\lambda(X,q,t)$ \cite{M1} form a family of $W$-invariant polynomials, where $W$ is a Weyl group of $\fg$, $X$ is a collection of $n={\rm rk}(\fg)$ variables $x_i$ and $\la\in P$ is a dominant weight. 
	These polynomials have several equivalent definitions and enjoy many nice properties.
	In particular, Macdonald polynomials form an orthogonal 
	basis of the ring of symmetric functions 
	for the constant term pairing and 
	constitute the eigenbasis for the action of a collection of $n$ pairwise commutative self-adjoint Macdonald operators.
	Certain specializations of Macdonald polynomials are very meaningful: for example, $q=t$ case corresponds to the Schur functions, $q=0$ leads to the Hall-Littlewood functions, 
	$q=t^\alpha$, $t\rightarrow 0$ are the Jack polynomials and $t=0$ substitution produces the $q$-Whittaker functions or $q$-Hermite polynomials (see \cite{M1, M2}). 
	All the specializations above are important from the point of view of representation theory. In particular, the 
	$q$-Whittaker functions have to do with characters of representations of current Lie algebras (see \cite{BBCKhL, CG, CI, CO1, CO2, I, Sa}).
	
	The nonsymmetric Macdonald polynomials $E_\lambda(X,q,t)$ (see \cite{M2}) are labeled by all weights $\la\in P$.
	This is another family of polynomials orthogonal for the (nonsymmetric)
	Cherednik pairing. The polynomials $E_\la(X,q,t)$  naturally show up in the
	theory of DAHA (double affine Hecke algebra) \cite{Ch1, Ch2, H}; in particular, they are eigenfunctions for the $Y$-generators of DAHA.
	The appropriate symmetrization maps the nonsymmetric pairing to the Macdonald pairing, the $Y$-generators of DAHA to Macdonald operators, and the nonsymmetric Macdonald polynomials to the symmetric ones.
	One of the most well-known outcomes from the nonsymmetric theory is the Cherednik computation of the norms of Macdonald polynomials.  
	Let us also mention that the $t=0$ and $t=\infty$ specializations of  $E_\la(X,q,t)$ have to do with representation theory of the Iwahori subalgebra
	of the affine Kac-Moody Lie algebra attached to $\fg$ (see \cite{FMO1,FMO2,FKM,FKhMO,LNSSS,NNS1,NNS2,NS}).
	
	In this paper, we deal with the parasymmetric (intermediate) Macdonald polynomials 
	(see \cite{Bar1,Bar2,G,Lap} for the $\gl_n$ case and a recent preprint \cite{Sch} for the general settings).
	To a finite subset $J$ of the set of simple roots $\Phi$ one assigns a parabolic Lie subalgebra $\fp_J\subset\fg$ and its Weyl subgroup $W^J$.
	The elements of the ring of $W^J$-symmetric polynomials are called parasymmetric polynomials. 
	The parasymmetric Macdonald polynomials $E_\lambda^{J}(X,q,t)$ labeled by $J$-antidominant weights $\la$ 
	are partial symmetrizations of the nonsymmetric Macdonald polynomials.
	These polynomials constitute an orthogonal basis for a certain (parasymmetric) pairing $\langle \cdot, \cdot\rangle^J$ 
	given by constant term type formulas (see Section \ref{Macdonald} for details) and are eigenfunctions of the partially symmetric Macdonald operators. 
	Our main goal is to categorify the specializations at $t=0$ and $t=\infty$ of the 
	parasymmetric Macdonald polynomials $E_\lambda^{J}(X,q,t)$ via representations of the parabolic subalgebras $\cP_J$ of the affine Kac-Moody Lie algebra $\widehat\fg$ (note that $\cP_J\supset \fp_J$). 
	In what follows we refer to $\cP_J$ as parahoric Lie (sub) algebras; they are parabolic subalgebras containing the Iwahori Lie algebra $\cI$.
	
	Let $w_0^J$ be the longest element in the Weyl group $W^J\subset W$.
	For a $W^J$-antidominant weight $\lambda$ we define certain graded cyclic $\cP_J$-modules $D_\lambda$ and $U_\lambda$ (see Definitions ~\ref{Ddefrel} and \ref{Udefrel}) and prove the following Theorem. 
	
	\begin{theoremA}
		\label{thm::intro::characters}
		(Proposition~\ref{prp::specialization}, Corollary~\ref{cor::DU::characters})
		The specialized parasymmetric Macdonald polynomials $E_{\lambda}^{J}(X,q,0)$ and $E_{\lambda}^{J}(X,q,\infty)$ are well defined, $E_{\lambda}^{J}(X,q,0)$
		is equal to the graded character of $D_\lambda$  and $E_{\lambda}^{J}(X^{-1},q^{-1},\infty)$ is equal to the graded character of $U_{-w_0^J\lambda}$. 
	\end{theoremA}	
	
	The $t=0$ specialization of the (parabolic) constant term pairing $\langle\cdot,\cdot\rangle^J$ and the modules $D_\lambda$ and $U_\lambda$ admit a transparent categorical description.
	We introduce the abelian category $\fC^{J}$ of graded finitely cogenerated $\cP_J$-modules  with finite-dimensional weight spaces and a "dual" category ${\fC^{J}}^{\vee}$ of graded finitely generated ${\cP^{J}}$-modules (see Section \ref{ParaCat} for the complete definition). 
	We show that the Euler characteristic of Ext's between modules in ${\fC^{J}}$ and ${\fC^{J}}^{\vee}$  categorify the parasymmetric pairing $\langle\cdot,\cdot\rangle^J_{t=0}$ (see Proposition \ref{prop::pairing::Ext}).
	In order to realize the modules $D_\la$ and $U_\la$ in the categorical terms we filter the abelian category $\fC^J$ (resp., ${\fC^{J}}^{\vee}$) by Serre subcategories $\fC^{J}_{\preceq\lambda}$ (resp., ${\fC^{J}}_{\preceq\lambda}^{\vee}$), where  
	$\preceq$ is the Cherednik partial order on weights.
	The categories $\fC^{J}_{\preceq\lambda}$ and ${\fC^{J}}_{\preceq\lambda}^{\vee}$
	consist of modules whose $\fh$-weights are less than or equal to $\lambda$. We prove the following Theorem.
	
	\begin{theoremA}
		\label{thm::intro::stratified}
		(Theorem \ref{catarestrat})
		The categories $\fC^J$ and ${\fC^{J}}^{\vee}$ are stratified. In particular, 
		for any modules $M_1,N_1\in \fC^{J}_{\preceq\lambda}$, $M_2,N_2\in {\fC^{J}}_{\preceq\lambda}^{\vee}$ one has
		$$
		\Ext_{\fC^{J}_{\preceq\lambda}}(M_1,N_1) = \Ext_{\fC^{J}}(M_1,N_1),\
		\Ext_{{\fC^{J}_{\preceq\lambda}}^{\vee}}(M_2,N_2) = \Ext_{{\fC^{J}}^{\vee}}(M_2,N_2).
		$$
	\end{theoremA}
	In other words, the derived functors of the fully faithful embeddings of the abelian categories $\fC^{J}_{\preceq\lambda}\hookrightarrow \fC^{J}$ and
	${\fC^{J}}^\vee_{\preceq\lambda}\hookrightarrow {\fC^{J}}^\vee$
	are fully faithful embeddings of corresponding derived categories
	(recall that in~\cite{FKhMO} we proved Theorem~\ref{thm::intro::stratified} for $J=\emptyset$,  $\cP^J=\cI$; the maximal case $\cP^{J}=\fg[z]$ is worked out in~\cite{Kh})
	
	The ${\rm Ext}$ property formulated in Theorem~\ref{thm::intro::stratified} naturally shows up in the theory of quasi-hereditary algebras ~\cite{Dlab}. The corresponding categories are called the highest weight categories ~\cite{CPS,Krause}. 
	The stratified categories generalize the notion of the highest weight categories ~\cite{Br,BS,Kh,Kl,FKhMO}
	(see Definition~\ref{StandardCostandard} for the precise definition). 
	Let $\overline{\Delta}_\la^J$ and  $\overline{\nabla}_\lambda^J$ be the proper standard and proper costandard modules in the categories ${\fC^{J^{\vee}}_{\preceq\lambda}}$ and $\fC^{J}_{\preceq\lambda}$. 
	We show that $\overline{\nabla}_\lambda^{J}$ is isomorphic to the graded dual of $U_{w_0^J\lambda}^{J}$ and $\overline{\Delta}_{\lambda}^{J}$ is isomorphic to $D_\lambda^{J}$.
	We also prove that the characters of $\overline{\Delta}_{\lambda}^{J}$ and $\overline{\nabla}_\lambda^{J}$ are obtained via 
	the Gram-Shmidt algorithm applied to the characters of irreducible modules.

	As an application of theorem~\ref{thm::intro::stratified} we are able to prove the analogue of the 
	Peter-Weyl-van der Kallen theorem (\cite{PW,vdK,TY}) for the parabolic and parahoric groups. Recall that the classical Peter-Weyl theorem describes the bi-module of
	functions on a simple Lie group as a direct sum of tensor products of the irreducible modules with their duals. One has a similar theorem due to van der Kallen for the space of functions on the Borel
	subgroup. The main difference is that the direct sum decomposition is no longer available and one has to consider an associated graded space for
	a natural filtration; the graded pieces of this filtration are the tensor products of the Demazure modules and the van der Kallen modules. 
	In  \cite{FKM,FMO1,FKhMO} we generalized the Peter-Weyl-van der Kallen theorem to the case of the current groups and the Iwahori groups. 
	The main new ingredient is that the tensor products as above have to be considered 
	over certain commutative (polynomial) algebras, the so-called highest-weight algebras. 
	
	Let $\mathcal{A}_\lambda$ be the highest weight 
	algebra (see \cite{FKhM,FKhMO,FMO1}) for the costandard module $\nabla^J_\la$ over $\cP_J$. Using Theorem \ref{thm::intro::stratified} 
	we derive the Peter-Weyl-van der Kallen type theorem for the parabolic and parahoric groups.
	
	\begin{theoremA} 
		\label{thm::intro::PeterWeyl} (Theorem \ref{PWvdKparahoric}, Corollary \ref{PWvdKparabolic}) 
		There exists a filtration  on the bi-module of algebraic functions on the parahoric Lie  group of 
		$\cP_J$  such that the associated graded
		space is isomorphic to the direct sum over $J$ antidominant weights $\lambda$ of the bi-modules
		$\nabla^J_\lambda \otimes_{\mathcal{A}_\lambda}((\Delta_\lambda^J)^\vee)^o$. 
	\end{theoremA}
	
	It is worth mentioning that  Theorems~\ref{thm::intro::stratified} and ~\ref{thm::intro::PeterWeyl} hold true for finite-dimensional parabolic Lie algebras $\fp\subset\fg$.
	Let $\fC^J_0$ (resp., ${\fC^{J}_0}^{\vee}$) be the category of finitely cogenerated (resp., finitely generated) $\fp_J$-modules  and let $\nabla^{J,fin}_\lambda$, $\Delta^{J,fin}_\lambda$ be the corresponding costandard and standard modules in these categories. 
	The following parabolic generalization of the van der Kallen theorem \cite{vdK} holds true.
	
	\begin{cor*} 
		(Theorem \ref{pararehwc}) The categories $\fC^{J}_0$ and ${\fC^{J}_0}^{\vee}$ are the highest weight categories.\\ 
		(Corollary \ref{PWvdKparabolic}) There exists a filtration of the bi-module of algebraic functions on the parabolic Lie  group of 
		$\fp_J$  such that the associated graded
		space is isomorphic to the direct sum over $J$ antidominant weights $\lambda$ of the bi-modules
		$\nabla^{J,fin}_\lambda \otimes ((\Delta_\lambda^{J,fin})^\vee)^o$.   
	\end{cor*}
	
	Let us finish the introduction with a few remarks. 
	
	First, as we pointed out above, the categories we are working with admit an extra structure depending on the order of the
	set of weights. All the theorems above are formulated for the standard Cherednik order. In the main body of the paper, we also work out the case of the dual Cherednik order; this case turned out to be more tricky.

	Second, the Peter-Weyl type theorems have much in common with
	the classical Howe duality for the general Lie algebras.
	The analogs of the Cauchy identities in the nonsymmetric settings were studied in \cite{FMO2,CK,KL,MN}. We expect that similar constructions are available
	in the parahoric case as well (see \cite{AE,AGL,Las} for partial results in this direction). We plan to return to this problem elsewhere.
	
	Third, we expect that analogs of the parasymmetric  Macdonald polynomials
	show up in a much more general situation.
	With each graded Lie algebra $\fa$ satisfying the assumptions given in Section~\ref{sec::parabolic::setup}
	we associate the pairing $\langle \cdot,\cdot \rangle^{\fa}$ that depends on two parameters $q$ and $t$:
	$$\langle f, g\rangle^{\fa} := [f g^{*} \mu^{\fa}_{q,t}]_1, \text{ where } 
	\mu^{\fa}_{q,t}:=
	\prod_{\alpha\in\Phi}\frac{(1-X^{\alpha})}{(1-tX^{\alpha})}\prod_{m\geq 0}\prod_{\gamma} \left(\frac{1-q^m X^{\gamma}}{1-t q^m X^{\gamma}}\right)^{\dim\fr(m)_\gamma}.
	$$
	
	\begin{conj*}
		Suppose that the categories $\fC(\fa)$ and $\fC^\vee(\fa)$ are stratified
		for a partial order $\prec$.
		Then there exist the polynomials $E_{\lambda}^{\fa}(X,q,t)\in \mathbb{Z}(q,t)[P]$, $\lambda \in P^-$, 
		such that $E_{\lambda}^{\fa}(X,q,t)\in \mathbb{Z}(q,t){\rm span}\{X^\nu|\nu \preceq \lambda\}$ and
		$\langle E_{\lambda}^{\fa}(X,q,t), E_{\mu}^{\fa}(X,q,t) \rangle^\fa=0,~ \text{if}~\lambda \neq \mu.$
		The specialization $E_{\lambda}^{\fa}$ at $t=0$ and $t=\infty$ produce the characters of the proper standard and proper costandard modules.
	\end{conj*}
	
	The paper is organized as follows.
	In Section \ref{sec::parabolic::setup} we collect basic definitions and notation on Lie algebras and describe the representations we are working with.
	In Section \ref{Macdonald} we define and study the parasymmetric Macdonald polynomials, show that they form an orthogonal basis of the ring of parasymmetric functions for the nonsymmetric and parasymmetric pairings, and focus on the specializations $t=0,\infty$.
	In Section \ref{ParaCat} we describe explicitly the categories of representations of $\cP_J$ we are working with and recall the general formalism of the stratified categories.
	In Section \ref{sec::standards} we introduce the cyclic modules $D_\lambda$ and $U_\lambda$, study their properties, and relate them to the standard and  costandard modules.
	Finally, in Section \ref{Main} we prove Theorems \ref{thm::intro::characters},~\ref{thm::intro::stratified} and \ref{thm::intro::PeterWeyl}.

	\section*{Acknowledgments}
	We are grateful to Daniel Orr and  Ivan Cherednik for useful discussions.

	\section{Lie algebras and their representations}
	\label{sec::parabolic::setup}
	In this section, we fix the notation and describe a class of Lie algebras containing parabolic and parahoric algebras. 
	\subsection{Parahoric Lie algebras}
	Let $\fg$ be a simple finite-dimensional Lie algebra of rank $n$, $\fb$ its Borel subalgebra, $\Phi=\Phi_+\cup \Phi_-$ its root system. 
	Let $\alpha_1, \dots, \alpha_n$ be the set of simple roots,
	and let $\omega_1,\dots,\omega_n$ be the fundamental weights.
	We denote by $P,P^\vee$ the weight lattices for $\fg$ and for the 
	Langlands dual Lie algebra $\fg^\vee$. Similarly, let  $Q$ and $Q^\vee$ be the corresponding root lattices.
	In particular,  $Q\subset P$; we set  $\Pi:=P/Q$.
	
	Let 
	$e_i=e_{\alpha_i}\in\fb$ be the root vectors of weight $\alpha_i$
	and let $f_j=f_{\alpha_j}$ be the root vectors of weight $-\alpha_j$. We denote by $W$ the Weyl group of $\fg$.

	With each subset $J\subset \{1, \dots, n\}$ we associate a parabolic subalgebra $\fp_{J}\subset \fg$ containing $\fb$, such that $f_i \in \fp_{J}$ if and only if $i \in J$. 
	We denote by $\Phi_{J}=\Phi_{J+}\cup \Phi_{J-}$ the root subsystem spanned by roots $\alpha_i$, $i \in J$ and let $\fl_J\subset\fp_J$ be the maximal reductive Levi subalgebra of $\fp_J$.
	
	Let
	\begin{equation}
		P^+_J:=\{\lambda \in P: \langle \lambda, \alpha_i^\vee \rangle\geq 0,~ i \in J\}, \quad
		P^-_J:=\{\lambda \in P: \langle \lambda, \alpha_i^\vee \rangle\leq 0,~ i \in J\}.    
	\end{equation}
	In other words, $P^+_J$ is the set of $\fp_J$ dominant weights, and $P^-_J$ is the set of $\fp_J$ antidominant weights. 
	If $J=\{1,\dots,n\}$, then we write simply $P^+$ and $P^-$ instead of $P^+_J$ and $P^-_J$.  
	One has
	\[
	P^+_J = \bigoplus_{i\in J} \bZ_{\ge 0} \omega_i \oplus \bigoplus_{i\notin J} \bZ \omega_i.
	\quad 
	P^-_J = \bigoplus_{i\in J} \bZ_{\le 0} \omega_i \oplus \bigoplus_{i\notin J} \bZ \omega_i.
	\] 
	We denote by $W^J$ the subgroup of $W$ generated by reflections in simple roots belonging to $J$, i.e.  $W^J:=\langle s_i, i \in J\rangle.$

	To a parabolic Lie subalgebra $\fp_J$ one assigns the corresponding parabolic Lie subalgebra $\cP_J$ in the affine Lie algebra $\gh$, the
	central extension of $\fg[z,z^{-1}]$ (see \cite{Kac}).
	The algebra $\cP_J$ corresponds to the same subset $J$ viewed as a subset of the nodes of the affine Dynkin diagram. One has
	\[
	\mathcal{P}_J = \fp_J \oplus z\fg[z].
	\]
	All the algebras $\cP_J$ contain the Iwahori subalgebra $\cI$ (the Borel subalgebra of $\gh$), this is why they are called parahoric subalgebras.
	One has two extreme cases: if $J=\varnothing$, then $\mathcal{P}_J$ is the Iwahori algebra $\cI$; if 
	$J=\{1,\dots,n\}$, then $\mathcal{P}_J$ coincides with the current algebra $\fg[z]$.
	In general, $\cI\subset \mathcal{P}_J\subset \fg[z]$ and 
	$\mathcal{P}_I\subset \mathcal{P}_J$ if and only if $I\subset J$.
	We denote by $\fr_J$ the radical of the parahoric Lie algebra $\cP_J$; in particular, we have the following (root) decompositions (recall $\fl_J\subset\fp_J\subset\fg$):
	\begin{gather*}
		\mathcal{P}_J = \fl_J + \cI = \fl_J \oplus \fr_J, \\ 
		\fl_J = \fh\oplus_{\alpha\in \Phi_J} \fg_{\alpha}, 
		\quad \fr_J = \oplus_{\alpha\in \Phi_{+}\setminus\Phi_{J+}}\fg_{\alpha}\bigoplus_{n>0}(\oplus_{\alpha\in\Phi} \gh_{\alpha+n\delta}\oplus \fh\otimes z^{n}).
	\end{gather*}
	Here $\delta$ denotes the basic imaginary root and $\fg_{\alpha+n\delta} = \fg_{\alpha}\otimes z^n$.
	
	In this paper, we mainly deal with representations of the parahoric algebras $\cP_J$.
	The main tool we use is the restriction of $\cP_J$ modules to the Iwahori algebra. Let us briefly recall the main classes of
	$\cI$ modules popping up in the following sections.
	
	For $\la\in P^-$ let $W_\la$ and $\W_\la$ be the corresponding local and global Weyl modules over the current algebra $\fg[t]$ (see \cite{BrFi,CFK,FL}. The global module $\W_\la$ admits a free action
	of  a certain
	polynomial algebra $\cA_\la$ (usually referred to as a highest weight algebra); the action commutes with the $\fg[t]$ action and
	the tensor product over $\cA_\la$ of $\W_\la$ with the trivial one-dimensional module is isomorphic to  $W_\la$. 
	The algebra $\cA_\la$ is identified with the endomorphism algebra ${\rm End}_{\fg[t]} \W_\la$.

	One can extend the family of Weyl modules to all $\la\in P$ (see \cite{FeMa,FMO1,Kat}. 
	The modules $W_\la$ and $\W_\la$ do not admit the action of $\fg[t]$ for $\la\notin P^-$, but rather are acted upon by the Iwahori algebra. The generalized Weyl modules are cyclic $\cI$-modules and are defined via the explicit set of relations.
	They share many nice properties with the classical Weyl modules. In particular, the global generalized Weyl modules admit a free 
	the action of the polynomial highest weight algebra which intertwines the $\cI$ action;  the quotient of the global modules by the augmentation ideal is identified with the local modules.
	
	In order to study the generalized Weyl modules one needs to consider
	a wider of cyclic Iwahori representations, the generalized Weyl 
	modules with characteristics $W_\la(r)$ and $\W_\la(r)$, $r\in\bZ_{\ge 0}$ (see \cite{FeMa,FMO2,FKM,FKhMO}). The modules
	$W_\la(r)$ and $\W_\la(r)$ are defined by the explicitly given
	set of relations; the relations depend on the reduced decomposition of an element from the 
	extended affine Weyl group and $r$ varies from $0$ to the length of this element. The generalized Weyl modules with characteristics are the main building blocks of the so-called 
	decomposition procedure: a filtration on the generalized Weyl modules which is governed by the combinatorics of the quantum alcove paths (see \cite{RY,OS,FeMa}). 
	More precisely, each space of the filtration 
	in the decomposition procedure, as well as each quotient, is identified with a generalized Weyl module with characteristics. 
	This construction allows one to establish a link to the theory of the nonsymmetric Macdonald polynomials.

	\subsection{General framework: Lie algebras, representations, and pairings}\label{GeneralFramework}
	In this subsection, following \cite{FKhMO} we describe a class of Lie algebras $\fa$ we are going to work with and certain categories 
	of their representations.
	Later, we will restrict to the case of parabolic and parahoric Lie algebras.
	
	Let $\fa$ be a $\bZ_{\ge 0}$ graded Lie algebra, $\fa=\bigoplus_{m\ge 0} \fa(m)$. We assume the following:
	\begin{itemize}
		\item $\fa=\fa_0\oplus \mathfrak{r}$, where $\fa_0\subset \fa(0)$ is a reductive Lie subalgebra with a Cartan subalgebra $\fh\subset\fa_0$ and the associated root system $\Phi$,
		\item $\mathfrak{r}=\bigoplus_{m\ge 0} \mathfrak{r}(m)$ is a graded ideal,
		\item each graded component $\mathfrak{r}(m)$, $m>0$, is a finite-dimensional $\fa_0$-module,
		\item $\dim \mathfrak{r}(0)<\infty$ and all the  $\fh$-weight spaces of $\U(\mathfrak{r}(0))$ are finite-dimensional,
		\item $\bigcap_{i=0}^\infty (ad\, \mathfrak{r})^i.\fa =0$,
		\item $\dim(\fa/(ad\, \mathfrak{r})^i.\fa)<\infty$ for any $i \geq 0$.
	\end{itemize}
	Let $\overline{\fa}\subset \fa$ be the $\fh$-weight zero subalgebra of $\fa$ (in particular, $\fa_0\supset\fh\subset\overline{\fa}$). 
	We assume that $\overline{\fa}$ is abelian.
	For example, if $\fa$ is a parahoric subalgebra, then $\overline{\fa}=\fh[t]$, which is clearly abelian.
	
	Let $\{\omega_i|\ i\in I\}$ be the set of fundamental weights in the weight lattice $P$ associated with the reductive Lie algebra $\fa_0$. Each graded component $\fr(m)$ is a finite-dimensional $\fa_0$-module, therefore, there exists a finite subset $S_m\subset P$ and the root decomposition:
	$$ \fr(m) = \sum_{\alpha\in S_m} \fr(m)_{\alpha}.$$
	
	To each algebra $\fa$ as above we assign two categories of $\fa$-modules $\fC(\fa)$ and $\fC^\vee(\fa)$:
	\begin{itemize}
		\item 
		$\fC(\fa)$ is the category of $\bZ$-graded $\fa$ modules $M=\bigoplus_{m\in\bZ} M(m)$ subject to the following conditions: 
		\begin{itemize}
			\item $M(m)=0$ for $m$ large enough,
			\item each $M(m)$ is a direct sum of finite-dimensional integrable irreducible $\fa_0$-modules with finite multiplicities,
			\item all the $\fh$-weight spaces of $\ker \mathfrak{r}^i$ are finite-dimensional for all $i> 0$,
			\item for any $v\in M$ there exists an $i>0$ such that $\mathfrak{r}^iv=0$, i.e. $\bigcup_{i=1}^\infty \ker \mathfrak{r}^i=M$.
		\end{itemize}
		\item
		The {"opposite"} category $\fC^\vee(\fa)$ consists of $\bZ$-graded $\fa$ modules $M=\bigoplus_{m\in\bZ} M(m)$ such that 
		\begin{itemize}
			\item $M(m)=0$ for $m$ small enough,
			\item each $M(m)$ is a direct sum of finite-dimensional integrable irreducible $\fa_0$ modules with finite multiplicities,
			\item $M$ is generated as an $\fr$ module by a subspace $\overline{M}$ admitting a decomposition into integrable irreducible $\fa_0$ modules with finite multiplicities,
			\item all the $\fh$-weight spaces of $M/\mathfrak{r}^i M$ are finite-dimensional for all $i>0$.
		\end{itemize}
	\end{itemize}
	The irreducible objects in categories $\fC(\fa)$ and $\fC^{\vee}(\fa)$ are irreducible $\fa_{0}$-modules placed in appropriate gradings, so they are indexed by $\Lambda:=P^-(\fa_0)\times \bZ$ and are denoted by $L_{\lambda,k}$ with $\lambda\in P^-(\fa_0)$, $k\in\bZ$.
	\begin{rem}
		We opt to parametrize the irreducible modules by the anti-dominant weights (as opposed to the more standard dominant weights)
		since it is more natural from the point of view of Macdonald polynomials.
	\end{rem}
	
	For each irreducible $\fa_0$-module $L_{\lambda}$ its character is a function on $\fh$ symmetric with respect to the Weyl group $W(\fa_0)$ of the Lie algebra $\fa_0$:
	$$[L_{\lambda}]:= \sum_{\alpha} \dim (L_{\lambda})_{(\alpha)}X^{\alpha} \in \bZ[P]^{W(\fa_0)}.$$
	For example, if $\fa_0$ is abelian, then the module $L_{\lambda}$ is one-dimensional and $[L_{\lambda}]=X^\lambda.$ 
	For each object $M\in\fC$ we can write down its character:
	\begin{equation}
		\label{eq::character}
		[M]:= \sum_{k\in\bZ}\sum_{\lambda\in P_{+}} [M:L_{\lambda,k}] q^k [L_{\lambda}] \in \bZ[q,q^{-1}]](P^+(\fa_0)).   
	\end{equation}
	Here by $[M:L_{\lambda,k}]$ we denote the multiplicity of the irreducible representation $L_{\lambda,k}$ in the module $M$.
	
	It is worth mentioning that for $M\in\fC(\fa)$ its character $[M]$ is 
	a polynomial in $q$ and a series in $q^{-1}$.
	For the opposite category, for each object $N\in\fC^{\vee}(\fa)$ its character $[N]$ written by the same formula~\eqref{eq::character} defines an element from $\bZ[[q,q^{-1}][P]$ which is a series in $q$ and a polynomial in $q^{-1}$.
	
	The category $\fC(\fa)$ has enough injectives, e.g. the coinduced modules 
	\[ \mathbb{I}_{\lambda,k}=\mathrm{coInd}_{\fa_0}^{\fa}L_{\lambda,k}=\Hom_{\U(\fa_0)}(\U(\fa),L_{\lambda,k}) \simeq L_{\lambda,k}\otimes \U(\fr)^\vee\]
	is an injective hull of an irreducible representation in $\fC(\fa)$.
	We assume that each module $M\in\fC(\fa)$ admits an injective resolution $M\hookrightarrow I^{\udot}$ such that for each weight $\lambda$ and integer $k$ the multiplicity of the irreducible $L_{\lambda,k}$ is finite in $I^{\udot}$.
	The corresponding derived category $\bD_{+}(\fC(\fa))$ consists of bounded from the left complexes of $\fa$-modules from $\fC(\fa)$ such that the total multiplicity of each irreducible $L_{\lambda,k}$ is finite. 
	The Euler characteristic of each object from $\bD_{+}(\fC(\fa))$ is well defined:
	$$
	\bD_{+}(\fC(\fa)) \rightarrow \bZ[q,q^{-1}]](P_{+}), \quad [M^{\udot}] \mapsto \sum_k (-1)^{k} [M^k],
	$$
	where $[M]$ stands for the character of the module $M$. 
	The analogous properties are assumed to be satisfied in $\fC^{\vee}(\fa)$ as well.
	First, we notice that the induced module 
	\[\mathbb{P}_{\lambda,k} = \mathrm{Ind}_{\fa_0}^{\fa}L_{\lambda,k}\simeq L_{\lambda,k}\otimes \U(\fr)
	\]
	belongs to the category $\fC^{\vee}(\fa)$ and is a projective cover of the irreducible representation $L_{\lambda,k}$.
	Second, we assume that each module $N\in\fC^{\vee}(\fa)$ admits a projective resolution $P^{\udot}\twoheadrightarrow N$ such that the multiplicity of any irreducible $L_{\lambda,k}$ in $P^{\udot}$ is finite.
	Third, we consider the derived category $\bD_{-}(\fC^{\vee}(\fa))$ generated by bounded from the right complexes such that the multiplicity of each irreducible is finite and, consequently, the Euler characteristic of an object from $\bD_{-}(\fC^{\vee}(\fa))$ is a well-defined series in $q$ and polynomial in $q^{-1}$.

	Let  $\star$ denote the involution $f(X,q)^{\star} = f(X^{-1},q^{-1})$ and let $[\varphi]_1$ be the constant term, i.e.
	the coefficient $\varphi_0$ in the expansion $\varphi = \sum_{\lambda\in P} \varphi_{\lambda} X^{\lambda}\in \bZ[P]$.
	Let 
	\begin{equation}
		\label{eq::kernel}
		\mu^{\fa}=\prod_{\alpha\in\Phi}(1-X^{\alpha})\prod_{m\geq 0}\prod_{\alpha} (1-q^m X^{\alpha})^{\dim\fr(m)_\alpha}.
	\end{equation}
	Finally, we denote by $M\{k\}$ the module $M$ with the grading shifted by $k$.
	
	\begin{thm}
		\label{thm::Ext::pairing}
		For all objects $M,N \in \fC$ one has an isomorphism of graded vector spaces:
		\begin{equation}
			\label{eq::Ext::Lie}
			\sum_{k\in\bZ} Ext_{\bD_{+}(\fC(\fa))}^{\udot}(M\{k\},N) \simeq H^{\udot}(\fa,\fa_{0};\Hom(M,N)).
		\end{equation}
		If, moreover, $M$ is finite-dimensional, isomorphism~\eqref{eq::Ext::Lie} leads to the following equality of Euler characteristic:
		\begin{multline}
			\label{eq::Ext::pairing}
			\sum_{i\geq 0} \sum_{k\in\bZ} (-1)^i q^{k} 
			\dim {\rm Ext}^i_{\bD_{+}(\fC(\fa))}(M\{k\},N) =
			\\
			= \sum_{i\geq 0}(-1)^{i}\dim_{q}H^{i}(\fa,\fa_{0};\Hom(M,N))
			= \left[[M]^{\phantom{\frac{1}{2}}} [N]^{\star} {\mu^{\fa}} \right]_{1}^{\star}
		\end{multline}
		Similarly, if $M,N\in\fC^{\vee}(\fa)$ and $N$ is finite-dimensional, then we have the following equality of Euler characteristics:
		\begin{multline}
			\label{eq::Ext::pairing::1}
			\sum_{i\geq 0} \sum_{k\in\bZ} (-1)^i q^{k} 
			\dim {\rm Ext}_{\bD_-(\fC^\vee(\fa))}^{i}(M\{k\},N) =
			\\
			= \sum_{i\geq 0}(-1)^{i}\dim_{q}H^{i}(\fa,\fa_{0};\Hom(M,N))
			= \left[[M]^{\phantom{\frac{1}{2}}} [N]^{\star} {\mu^{\fa}} \right]_{1}^{\star}.
		\end{multline}
	\end{thm}
	\begin{proof}
		The proof for $\fa=\fg[t]$ can be found in \cite{Kh} and for $\fa=\cI$ 
		in \cite{KKhM}. The proof of the general case follows the same steps. For the reader's convenience, we recall the main arguments below.
		
		The isomorphism~\eqref{eq::Ext::Lie}    
		between the extension groups and relative Lie algebra cohomology follows from the fact that both functors are higher derived functors with the same abelian one.
		One should deal with relative cohomology because all modules in $\fC$ are direct sums of finite-dimensional $\fa_0$-modules.
		
		Recall that the Chevalley-Eilenberg complex 
		$$
		C^{\udot}_{CE}(\fa,\fa_{0};M) = 
		\Hom_{\fa_0}(\Lambda^{\udot} \fa/\fa_{0}; M) =
		\Hom_{\fa_0}(\Lambda^{\udot} \fr; M)
		$$
		is the complex computing of the relative Lie algebra cohomology.
		To explain equality~\eqref{eq::Ext::pairing} it remains to compute the Euler characteristic of the Chevalley-Eilenberg complex.
		The character of $\Hom(M,N)\simeq M^{\vee}\otimes N$ is equal to $[M]^*\cdot [N]$, the character of the exterior power is equal to the product 
		$$
		[\Lambda^{\udot}(\fr)] =[\otimes_{m\geq 0}\Lambda^{\udot}(\fr(m))] = \prod_{m\geq 0}\prod_{\alpha} (1-q^m X^{\alpha})^{\dim\fr(m)_\alpha}.
		$$
		The negative sign corresponds to the homological grading that comes with the exterior power.
		Finally, the dimension of $\fa_0$-invariants in a $\fa$-module $M$ coincides with the coefficient of $1=X^{0}$ in the expression 
		$\prod_{\alpha\in\Phi}(1-X^{\alpha}) [M]$.
	\end{proof}

	\begin{prop}
		The parahoric Lie algebras $\cP_J$ satisfy the assumptions for a Lie algebra $\fa$ as above.
	\end{prop}
	\begin{proof}
		The assumptions on the grading of the parahoric Lie algebra and its radical are satisfied because each graded component is a submodule of the finite-dimensional representation $\fg\otimes z^k$.
		
		The projective modules $\bP_{\lambda,k}$ are isomorphic to parabolic Verma modules $\Ind_{\cP_J^{-}}^{\gh} L_{\lambda}$.
		Thanks to the parabolic version of the BGG-resolution (see e.g.~\cite{Kum}) we know that each irreducible module $L_{\lambda,k}$ 
		admits a projective resolution that belongs to $\bD_-(\fC^\vee(\cP_J))$ and each module from $\fC^\vee(\cP_J)$ admits a projective resolution with finite multiplicity of each irreducible.
		By taking graded duals one gets the injective resolutions of modules.
	\end{proof}
	
	The kernel $\mu_{\cP_{J}}$ in the Ext-pairing~\eqref{eq::kernel} is given by
	\begin{equation}
		\label{eq::kernel::P_J}
		\mu_{\cP_{J}}= \prod_{\alpha\in\Phi_{J-}\sqcup\Phi_{+}}(1-X^{\alpha}) \prod_{k=1}^{\infty}\left(\prod_{\alpha\in\Phi}(1-q^{k}X^{\alpha})\right) \prod_{k=1}^{\infty}(1-q^k)^n.
	\end{equation}
	The last factor $\prod_{k=1}^{\infty}(1-q^k)^n$ does not depend on $X$ and hence does not play a serious role.

	\section{Parabolic pairings and orthogonal polynomials}\label{Macdonald}
	In this section, we study the parasymmetric Macdonald polynomials (a.k.a. intermideate Macdonald polynomials, see \cite{Bar1,Bar2,Sch}).  
	We use the notation from Section \ref{sec::parabolic::setup} and refer to \cite{BB} for the results about the Weyl groups.

	\subsection{Double affine Hecke algebra and nonsymmetric Macdonald pooynomials}

	Let $W^{af}=W\ltimes Q^\vee$ be the affine Weyl group and let $W^{e}=W\ltimes P^\vee$ be the extended affine Weyl group.
	Note that $\Pi=P/Q$ acts by the automorphisms on the affine
	Dynkin diagramm of $\fg$; it
	also naturally acts on the lattice $P$.
	Recall that the set of weights admits the standard partial \emph{dominance order} "$\leq$":
	$$
	\lambda \leq \mu \Longleftrightarrow \mu = \lambda +\sum_{i} k_i \alpha_i, \text{ with } k_i\geq 0.
	$$
	The Weyl group $W$ admits the standard Bruhat order "$\leq$" generated by inequalities:
	$$\sigma < s_\gamma \sigma\ \text{ iff } \ l(s_\gamma\sigma) = l(\sigma)+1, ~\gamma \in \Delta.$$
	For $\lambda \in P$ we denote by $\lambda_-$ the antidominant element in the $W$-orbit of $\lambda$ and let
	$\sigma_\la$ be the shortest element in $W$ such that $\lambda =\sigma_\la (\lambda_-)$.
	
	\begin{dfn}\label{def::Cherednik::ordering}
		The \emph{Cherednik partial order} "$\preceq$" and the \emph{dual Cherednik partial order} "$\preceq^{\vee}$" on the weight lattice $P$ are defined as follows 
		\[
		\begin{array}{r}
			\lambda\preceq \mu 
		\end{array}
		\Longleftrightarrow
		\left[
		\begin{array}{l}
			\lambda_- > \mu_-\\
			\lambda_- = \mu_- ~\text{and}~ \la \geq \mu
		\end{array}
		\right. 
		\Longleftrightarrow
		\left[
		\begin{array}{l}
			\lambda_- > \mu_-\\
			\lambda_- = \mu_- ~\text{and}~ \sigma_\la \geq \sigma_\mu 
		\end{array}
		\right. 
		\]
		\[
		\lambda\preceq^\vee \mu \Longleftrightarrow
		\left[
		\begin{array}{l}
			\lambda_- >  \mu_-\\
			\lambda_- = \mu_- ~\text{and}~ \la \leq \mu
		\end{array}
		\right. 
		\Longleftrightarrow
		\left[
		\begin{array}{l}
			\lambda_- > \mu_-\\
			\lambda_- = \mu_- ~\text{and}~ \sigma_\la \leq \sigma_\mu 
		\end{array}
		\right. 
		.
		\]
	\end{dfn}
	
	Let $e$ be the smallest number such that $\langle P, P^\vee \rangle\in \frac{1}{e}\bZ$.
	We consider the lattice $P \oplus \mathbb{Z}\frac{1}{e}\delta$, where $\langle \delta,P \rangle=\langle \delta, \delta \rangle=0$.
	There exists the natural action of the affine Weyl group $W^{af}$ on this lattice. 
	
	We consider the group ring $\mathbb{Z}(q^{\frac{1}{2e}},t)[P]$, where $q,t$ are formal variables. 
	Denote by $X^\lambda$ the formal exponential of the weight $\lambda$, $X_i:=X^{\omega_i}$, $X^\delta=q$. The action of $W^{af}$ is extended to
	this ring by $\tau X^{\mu}=X^{\tau(\mu)}.$
	The ring $\mathbb{Z}(q^{\frac{1}{2e}},t)[P]$ admits an action of the Double Affine Hecke Algebra $\DAHA$. 
	It is generated by elements  $\hX^\lambda$, $\lambda \in P$, $\pi \in \Pi$ and $\hT_i$, $i=0, 1, \dots, n$ satisfying the following relations (here $\hX_i=\hX^{\omega_i}$):
	
	\begin{gather}
		\hX^{\lambda} \hT_i = \hT_i \hX^{\lambda}, \text{ if } \langle\alpha_i^\vee,\lambda\rangle =0;\nonumber \\
		\hT_i\hX^{\lambda} \hT_i = t  \hX^{s_i\lambda}, \text{ if } \langle\alpha_i^\vee,\lambda\rangle =1; \nonumber\\
		\label{eq::Hecke::alg}
		\hT_i^{-1} = t^{-1} \hT_i + t^{-1} -1 \Leftrightarrow (\hT_i+1)(\hT_i-t) =0; \nonumber\\
		\label{eq::Braid::Relations}
		\underbrace{\hT_i \hT_j \hT_i \ldots }_{m_{{i,j}} \text{ factors}} \simeq \underbrace{\hT_j  \hT_i \hT_j \ldots }_{m_{{i,j}}\text{ factors}} \text{ with }m_{{i,j}} = ord_{W^{af}}(s_is_j);\\
		\pi \hT_i \pi^{-1}=\hT_j,\text{ if } \pi(\alpha_i)=\alpha_j;\nonumber\\
		\pi \hX^{\la} \pi^{-1}=\hX^{\pi(\la)}.\nonumber
	\end{gather}
	
	We denote by $\mathbb{Z}(q^{\frac{1}{2e}},t)[\hX]$ the ring generated by the elements $\hX_i^{\pm 1}$, $i=1,\dots, n$. Clearly, this ring admits an action of $W$ and it acts on $\mathbb{Z}(q^{\frac{1}{2e}},t)[P]$ as on the free module of rank one.
	
	The Double Affine Hecke Algebra was introduced by Cherednik, we use a slight modification of the generators compared to the initial definition.
	Namely, we rescale the generators $\hT_{s_i}$ by $\sqrt{t}$ and make the quadratic relations less symmetric but without half-integer powers (see e.g.~\cite{Ch2}).

	\begin{rem}
		Denote $\hT_{s_i}:=\hT_i$. 
		Relations \eqref{eq::Braid::Relations} imply the existence of elements $\hT_w$ for any $w$ from the affine Weyl group.  \end{rem}

	The action of  $\DAHA$ on $\mathbb{Z}(q^{\frac{1}{2e}},t)[P]$ is defined by the
	formulas
	\begin{gather*}
		\hX^\lambda X^\mu=X^{\lambda+\mu},\\
		\pi X^{\la}=X^{\pi(\la)},\\
		\hT_i X^\mu=\left(t s_i+\frac{t-1}{X^{\alpha_i}-1}(s_i-1)\right)X^\mu,
	\end{gather*}
	where $X^{\alpha_0}=qX^{-\theta}$ for the highest root $\theta$.
	
	For any $\nu \in P^\vee$ recall the element $\hY^{\nu}\in\DAHA$. These elements are defined in the following way. 
	Consider the translation element ${\mathbf t}_{\nu_+} \in W^{e}$, $\nu_+ \in P^{\vee +}$ and its reduced decomposition ${\mathbf t}_{\nu_+}=\pi s_{i_1}\dots s_{i_l}$. Then
	\begin{equation}
		\hY^{\nu_+}:= \pi \hT_{i_1}\dots \hT_{i_l}=\hT_{{\mathbf t}_{\nu_+}}.
	\end{equation}
	For any $\nu=\nu_+-\nu_- \in P^\vee$ let 
	$\hY^{\nu}:=\hY^{\nu_+}(\hY^{\nu_-})^{-1}$.
	These elements pairwise commute, $\hY^\eta \hY^\nu=\hY^{\eta+\nu}$, $\eta, \nu \in P^\vee$ (see \cite{Ch1}) and
	\begin{gather}\label{YTRelations}
		\hT_i \hY^\nu =  \hY^\nu \hT_i,~\text{if } \langle \alpha_i,\nu \rangle =0,\nonumber\\
		{\hT_i^{-1} \hY^\nu \hT_i^{-1}= t^{-1}\hY^{s_i\nu},~\text{if } \langle \alpha_i,\nu \rangle =1.}    
	\end{gather}
	
	Let $\mathbb{Z}(q^{\frac{1}{2e}},t)[\hY]$ be the ring generated by
	$\hY^{\omega_i^\vee}$, $i=1,\dots,n$. 
	The action of $W$ on this ring is defined by $\tau \hY^\nu=\hY^{\tau(\nu)}$.
	Let $\mu$ denote the nonsymmetric Cherednik kernel
	\[\mu=\prod_{k=1}^\infty \prod_{\alpha \in \Phi}\frac{1-q^k X^\alpha}{1-tq^k X^\alpha} \prod_{\alpha \in \Phi_+}\frac{1-X^\alpha}{1-tX^\alpha}.\] 
	For $f=f(X_1, \dots, X_n; q,t) \in\mathbb{Z}(q^{\frac{1}{2e}},t)[P]$ let
	\begin{equation}
		f^\star:= f(X_1^{-1}, \dots, X_n^{-1}; q^{-1},t^{-1}). 
	\end{equation}
	The map $\star$ is equivariant with respect to the antiautomorphism of $\DAHA$ defined on the generators in the following way:
	\[\hT_i^{\star}:=\hT_i^{-1},~(\hX^\lambda)^\star:=\hX^{-\lambda}, ~\pi^\star:=\pi^{-1},~q^{\star}:=q^{-1},~t^{\star}:=t^{-1}.\]
	For a series $f$ in variables $X_i$ we denote by $[ f]_1$ the constant term of $f$ and by $\mu_0$ the ratio $\mu/[ \mu]_1$.

	\begin{dfn}
		For $f,g \in \mathbb{Z}(q^{\frac{1}{2e}},t)[P]$ the Cherednik pairing is defined by
		$\langle f,g \rangle:=[f g^{*} \mu_0 ]_1.$
	\end{dfn}
	The following lemma is well known.
	\begin{lem}\label{Unitarity}
		\[\langle \hT_i^{-1}(f),g \rangle=\langle f,\hT_i(g) \rangle;~\langle \pi(f),g \rangle=\langle f,\pi^{-1}(g) \rangle.\]
		\end{lem}

	Recall that the nonsymmetric Macdonald polynomials $E_\la(X,q,t)$, $\la\in P$ are defined by the following two properties:
	\begin{gather*}
		E_\la(X,q,t)= X^\la +\sum_{\nu\prec\la} c_{\la,\nu}(q,t) X^\nu,\\
		\langle E_\la(X,q,t), E_\nu(X,q,t)\rangle = \delta_{\la,\nu}.
	\end{gather*}
	We note that $E_\la(X,q,t)$ are joint eigenfunctions of the operators $\hY^\nu$:
	\begin{equation}\label{NonsymmetricEigenvalues}
		\hY^\nu E_\la(X,q,t)= q^{-(\la,\nu)} t^{(\sigma_\la^{-1}(\rho^\vee),\nu)} E_\la(X,q,t).
	\end{equation}
	The formula \eqref{NonsymmetricEigenvalues} implies the orthogonality. Indeed, by Lemma \ref{Unitarity} we have
	\begin{equation}
		\label{Hermitian}    
		\langle \hH(f),g \rangle=\langle f,\hH^\star(g) \rangle
	\end{equation}
	for any $\hH \in \DAHA$.
	In particular, $\langle \hY^\nu(f),g \rangle=\langle f,\hY^{-\nu}(g) \rangle.$
	Now the eigenvectors of operators $\hY^\nu$ with different eigenvalues are orthogonal with respect to the Cherednik pairing.

	\subsection{Parabolic Cherednik order and parasymmetric Macdonald polynomials}
	Recall the natural action of $W$ on the ring $\mathbb{Z}(q^{\frac{1}{2e}},t)[P]$. For a subset $J \subset \{1, \dots,n\}$ recall the subgroup $W^J \subset W$. 
	In this section we study the subring of $\mathbb{Z}(q^{\frac{1}{2e}},t)[P]$ consisting of   $W^J$ invariant functions $\mathbb{Z}(q^{\frac{1}{2e}},t)[P]^{W^J}$. 
	Clearly the elements 
	\begin{equation}
		m^J_\lambda:=\frac{1}{|{\rm Stab}_{W^J}(\lambda)|}\sum_{\sigma \in W^J}X^{\sigma(\lambda)}, ~\la \in P^-_J
	\end{equation}
	form a $\mathbb{Z}(q^{\frac{1}{2e}},t)$ basis of $(\mathbb{Z}(q^{\frac{1}{2e}},t)[P])^{W^J}$.
	
	\begin{rem}
		We note that the polynomials $m_\la^J$ are defined for $J$-antidominant weights $\la$.
		If $J=\varnothing$, then $m_\la^J$
		are monomials. If $J=\{1,\dots,n\}$ (the maximal case), then $m_\la^J$
		are the monomial symmetric functions.
	\end{rem}
	
	We now define a parasymmetric pairing on $\mathbb{Z}(q^{\frac{1}{2e}},t)[P]^{W^J}$. Consider 
	\[
	\mu^J=\prod_{k=1}^\infty \prod_{\alpha \in \Phi}\frac{1-q^k X^\alpha}{1-tq^k X^\alpha} \prod_{\alpha \in \Phi_+}\frac{1- X^\alpha}{1-t X^\alpha}\prod_{\alpha \in \Phi_{J-}}\frac{1- X^\alpha}{1-t X^\alpha}\]
	(we note that only the last factor depends on $J$).
	Recall the normalization 
	$\mu_0^J:=\mu^J/[\mu^J]_1$ of $\mu^J$.
	
	\begin{dfn}
		\label{def::pairing::J}
		For $f,g \in \mathbb{Z}(q^{\frac{1}{2e}},t)[P]^{W^J}$ let
		$\langle f,g \rangle^J:=[ f g^\star \mu_0^J ]_1.$
	\end{dfn}

	\begin{dfn}\label{Parasymmetric}
		Parasymmetric Macdonald polynomials $E_\lambda^J(X_1, \dots, X_n; q,t)\in \mathbb{Z}(q^{\frac{1}{2e}},t)[P]^{W^J}$ are defined by the following properties:
		\begin{itemize}
			\item $E_\lambda^J(X_1, \dots, X_n; q,t)=m^J_\lambda+\sum_{\mu \prec \lambda}c_\la^\mu m^J_\mu$,
			\item $\langle E_\lambda^J(X_1, \dots, X_n; q,t), E_\mu^{J}(X_1, \dots, X_n; q,t) \rangle^J =0$ for $\lambda \neq \mu$.
		\end{itemize}
	\end{dfn}
	
	It is clear from the definition that $E_\lambda^J(X_1, \dots, X_n; q,t)$ are defined uniquely if exist.
	We consider the $J$ symmetrizer
	\[{\bf P}^J:=\sum_{\sigma \in W^J}\hT_\sigma.\]
	
	\begin{lem}\label{PSymmetrizer}
		For any  $f \in \mathbb{Z}(q^{\frac{1}{2e}},t)[P]$ the polynomial 
		${\bf P}^Jf$ is $W^J$-symmetric.
	\end{lem}
	\begin{proof}
		Let $\sigma \in W^J$ and assume $l(s_i\sigma)=l(\sigma)+1$ for some simple reflection $s_i \in W^J$. Then
		\[(\hT_\sigma+\hT_{s_i \sigma})f=(1+\hT_i)\hT_\sigma f.\]
		However
		\[(1+\hT_i)f=(1+s_i)\left( \frac{(t-1)X^{\alpha_i}}{1-X^{\alpha_i}} +t \right)f.\]
		Therefore $(\hT_\sigma+\hT_{s_i \sigma})$ is left divisible by $(1+s_i)$ and hence $(\hT_\sigma+\hT_{s_i \sigma})f$ is $W^{\{i\}}$
		symmetric (here $W^{\{i\}}=\{1,s_i\}$). However $W^J=\{\sigma \in W^J|\ l(s_i\sigma)=l(\sigma)+1\}\cup\{s_i\sigma \in W^J|\ l(s_i\sigma)=l(\sigma)+1\}$. Therefore  ${\bf P}^Jf$ is $W^{\{i\}}$
		symmetric for any $i \in J$, and hence  ${\bf P}^Jf$ is $W^J$-symmetric.
	\end{proof}
	
	\begin{lem}\label{PCommutes}
		For any  $\mathcal{G}_1 \in \mathbb{Z}(q^{\frac{1}{2e}},t)[\hX]^{W^J}\subset \DAHA$ and
		$\mathcal{G}_2 \in \mathbb{Z}(q^{\frac{1}{2e}},t)[\hY]^{W^J}\subset\DAHA$ one has: 
		\[{\bf P}^J\mathcal{G}_1=\mathcal{G}_1{\bf P}^J,\qquad 
		{\bf P}^J\mathcal{G}_2=\mathcal{G}_2{\bf P}^J.\]  
	\end{lem}
	\begin{proof}
		Two claims are equivalent because of $\DAHA$ duality (see \cite{Ch1}). We prove the first claim. 
		By assumption the polynomial $\mathcal{G}_1=\mathcal{G}_1(\hX_1^{\pm 1}, \dots, \hX_n^{\pm 1})$
		is symmetric with respect to $W^{\{i\}}$ for any $i\in J$. 
		Therefore
		\[\mathcal{G}_1\in\mathbb{Z}(q^{\frac{1}{2e}},t)[\hX_1^{\pm 1}, \dots, \hX_{i-1}^{\pm 1},\hX^{\omega_i}+\hX^{\omega_i-\alpha_i},\hX_{i+1}^{\pm 1}, \dots, \hX_{n}^{\pm 1}].\]  
		All the generators of this ring commute with $\hT_i$ and hence $\mathcal{G}_1\hT_i=\hT_i\mathcal{G}_1$ for $i \in J$.
		This implies the desired statement.
	\end{proof}
	
	\begin{lem}
		For any  $\mathcal{G} \in \mathbb{Z}(q^{\frac{1}{2e}},t)[\hY]^{W^J}$,
		$\la\in P_J^-$ one has: 
		\[\mathcal{G}m_\lambda^J=\sum_{P_J^-\ni \nu \preceq \lambda}c_{\mathcal{G},\lambda}^\nu m_\nu^J.\]
	\end{lem}
	\begin{proof}
		Clearly 
		\[\mathcal{G}m_\lambda^J=\mathcal{G}{\bf P}^Jm_\lambda^J={\bf P}^J \mathcal{G}m_\lambda^J \in \mathbb{Z}(q^{\frac{1}{2e}},t)[X]^{W^J}\]
		so $\mathcal{G}m_\lambda^J=\sum_{\nu}c_{\mathcal{G},\lambda}^\nu m_\nu^J.$
		We need to check that $c_{\mathcal{G},\lambda}^\nu=0$ if $\nu \not \preceq \lambda$. Assume that $c_{\mathcal{G},\lambda}^\nu\neq 0$ for some $\nu\not \preceq \lambda$. Then
		\[\mathcal{G}m_\lambda^J=\sum_{\eta \in P} d_\eta X^\eta\]
		and $d_\nu=c_{\mathcal{G},\lambda}^\nu\neq 0$. However $\nu\not \preceq \lambda$, which contradicts the $J=\varnothing$ case (see \cite{Ch1}).
	\end{proof}

	We denote
	\[R_\lambda^J(t):=\sum_{\tau \in {\rm Stab}_{W^J}({w_0^J\lambda})}t^{l(\tau)},\]
	where $\rm{Stab}_{G}(\xi)$ is the subgroup of a group $G$ stabilizing an element $\xi$ and  $w_0^J$ is the longest element in $W^J$.
	
	\begin{thm}\label{PSMc}
		Parasymmetric Macdonald polynomials  $E_\lambda^J(X_1, \dots, X_n; q,t)$ exist for any $\lambda \in P^-_J$. One has an explicit formula
		\[E_\lambda^J(X_1, \dots, X_n; q,t)=\frac{1}{R_\lambda^J(t)}{\bf P}^J E_{w_0^J\lambda}(X_1, \dots, X_n; q,t).\]
	\end{thm}
	\begin{proof}
		Using Lemma \ref{PSymmetrizer} we get 
		\[{\bf P}^J E_{w_0^J\lambda}(X_1, \dots, X_n; q,t)\in \bZ(q^{\frac{1}{2e}},t)[X]^{W^J}.\]
		For any  $\mathcal{G} \in \mathbb{Z}(q^{\frac{1}{2e}},t)[\hY]^{W^J}$:
		\[\mathcal{G} {\bf P}^J E_{w_0^J\lambda}(X_1, \dots, X_n; q,t)={\bf P}^J \mathcal{G} E_{w_0^J\lambda}(X_1, \dots, X_n; q,t)\]
		(see Lemma \eqref{PCommutes}).
		Hence ${\bf P}^J E_{w_0^J\lambda}(X_1, \dots, X_n; q,t)$ is an eigenvector for  $\mathcal{G}$. Let us compute the eigenvalue for $\mathcal{G}=\sum_{w \in W^J}\hY^{w \nu}$.
		Using equation \eqref{NonsymmetricEigenvalues} we get:
		\[
		\sum_{w \in W^J}\hY^{w(\nu)}{\bf P}^J E_{w_0^J\lambda}(X_1, \dots, X_n; q,t)=\Bigl(\sum_{w \in W^J}q^{-(w^{-1}(\la),\nu)} t^{(\sigma_\la^{-1}(\rho^\vee),w(\nu))}\Bigr){\bf P}^J E_{w_0^J\lambda}(X_1, \dots, X_n; q,t).\]
		Now let us consider the vectors of eigenvalues:
		\[
		S(\la)= S(\la)_\nu=
		\Bigl\{\sum_{w \in W^J}q^{-(w^{-1}(\la),\nu)} t^{(\sigma_\la^{-1}(\rho^\vee),w(\nu))}\Bigr\}_\nu.
		\]
		We note that  $S(\la)\ne S(\la')$ provided  $\la$ and $\la'$ are in different $W^J$ orbits (one sees this already for $t=1$). Then using equation \eqref{Hermitian} we have 
		\[\langle {\bf P}^J E_{w_0^J\lambda}(X_1, \dots, X_n; q,t), {\bf P}^J E_{w_0^J\nu}(X_1, \dots, X_n; q,t) \rangle =0, \]
		where $\langle \cdot, \cdot \rangle$ is the standard Cherednik pairing and $\la$ and $\nu$ are in different $W^J$ orbits.
		
		Note that $w( f g^\star \mu_0 )= f g^\star w(\mu_0) $ for $w\in W^J$, $f,g \in \mathbb{Z}(q^{\frac{1}{2e}},t)[P]^{W^J}$. Therefore $\langle f g^\star \mu_0 \rangle=\langle f g^\star w(\mu_0) \rangle$ and thus
		\begin{gather*}
			\Bigl[ {\bf P}^J E_{w_0^J\lambda}(X_1, \dots, X_n; q,t){\bf P}^J E_{w_0^J\nu}(X_1, \dots, X_n; q,t)w (\mu_0) \Bigr]_1 =0,~w\in W^J,\\
			\Bigl[ {\bf P}^J E_{w_0^J\lambda}(X_1, \dots, X_n; q,t){\bf P}^J E_{w_0^J\nu}(X_1, \dots, X_n; q,t)\sum_{w \in W^J}w (\mu_0) \Bigr]_1 =0.
		\end{gather*}
		However,
		$\sum_{w \in W^J}w (\mu_0)={\rm const}. \mu_0^J$ (see \cite{M2}). Hence the second property from Definition \ref{Parasymmetric} holds for the polynomials
		${\bf P}^J E_{w_0^J\lambda}(X_1, \dots, X_n; q,t)$.
		
		Now we are left to show that
		\begin{equation}\label{SymmetricTriangular}
			\frac{1}{R_\lambda^J(t)}
			{\bf P}^J E_{w_0^J\lambda}(X_1, \dots, X_n; q,t) = 
			m^J_{w_0^J\lambda}+\sum_{\nu \prec \lambda}c_\la^\nu m^J_\nu.
		\end{equation}
		Recall that ${\bf P}^J=\sum_{\sigma\in W^J} \hT_\sigma$. Let us apply $\hT_i$ to a monomial 
		$X^\nu$. Since $\hT_i X^\nu = (ts_i +\frac{t-1}{X^{\al_i}-1}(s_i-1))X^\nu$ we have $\hT_i X^\nu =\sum_{\eta \preceq \nu_-}d_{\eta} X^\eta$ for some constants $d_{\eta}$. Therefore
		\[
		{\bf P}^J E_{w_0^J\lambda}(X_1, \dots, X_n; q,t) = \sum_{\eta \preceq \lambda} d_{\eta}^\lambda X^\eta.
		\]
		Since ${\bf P}^J E_{w_0^J\lambda}(X_1, \dots, X_n; q,t)$ is $W^J$-symmetric we have
		\[
		{\bf P}^J E_{w_0^J\lambda}(X_1, \dots, X_n; q,t) = 
		c_\lambda^\lambda m^J_\lambda+\sum_{\nu \prec \lambda}c_\la^\nu m^J_\nu,
		\]
		so it suffices to show that $c_\la^\la=d_\la^{w_0^J\la}=R_\la^J(t)$.
		
		Let us consider a summand $\hT_w$ of ${\bf P}^J$. The $\DAHA$ element $\hT_w$ 
		is a product of the elements $\hT_i$ for $s_i\in {\rm Stab}_{W_J}(w_0^J\la)$ if and only if $w \in {\rm Stab}_{W_J}(w_0^J\la)$. If $\nu \preceq \lambda$, $\nu \neq w_0^J\la$, then for any $i \in J$
		the coefficient of $X^{w_0^J\la}$ in
		$\hT_i (X^\nu)$ is equal to $0$. 
		If $(\alpha_i,w_0^J\lambda)>0$ (i.e. $s_i\notin {\rm Stab}_{W^J}(w_0^J\la)$), then 
		$\hT_iX^{w_0^J\la}$ does not contain the term $X^{w_0^J\la}$. Therefore 
		if $w\notin {\rm Stab}_{W_J}(w_0^J\lambda)$, then $\hT_w X^{w_0^J\la} \in  \mathrm{span}\{X^\nu:\ \nu \preceq \la,\ \nu\neq w_0^J\la\}$.
		
		If $(\alpha_i,w_0^J\lambda)=0$ (i.e. $s_i\in {\rm Stab}_{W^J}(w_0^J\la)$), then 
		$\hT_iX^{w_0^J\la} \in tX^{w_0^J\la}+\mathrm{span}\{X^\nu:\ \nu \preceq \la,\ \nu\neq w_0^J\la\}$.
		Thus if $w\in {\rm Stab}_{W^J}(w_0^J\lambda)$, then $\hT_w X^{w_0^J\la} \in  t^{l(w)}X^{w_0^J\la} + \mathrm{span}\{X^\nu:\ \nu \preceq \la,\ \nu\neq w_0^J\la\}$. 
		This proves \eqref{SymmetricTriangular} and hence completes the proof of our Theorem.
	\end{proof}
	
	\begin{dfn}
		For a Weyl group $W^J$ and its parabolic subgroup $W^{K} \subset W^{J}$ we denote by $(W^{K}\backslash W^J)^{\min}$ the set of minimal length representatives of the right cosets,  $(W^{J}/W^{K})^{\min}$ the set of minimal length representatives of the left cosets.
	\end{dfn}
	
	The proof of Theorem \ref{PSMc} implies:
	\begin{cor}\label{ShortSymmetrizer}
		\begin{equation}
			\label{eq::sym::Macdonald}
			E_\lambda^J(X_1, \dots, X_n; q,t)=\sum_{\sigma \in (W^J/{\rm Stab}_{W^J} (w_0^J{\lambda}))^{\min}}\hT_\sigma E_{w_0^J\lambda}(X_1, \dots, X_n; q,t).
		\end{equation}
	\end{cor}

	\begin{cor}
		The polynomials $E_\lambda^J(X_1, \dots, X_n; q,t)$ are eigenfunctions for operators 
		$\mathcal{G} \in \mathbb{Z}(q^{\frac{1}{2e}},t)[\hY]^{W^J}$. Moreover
		\[
		\Bigl(\sum_{w \in W^J}\hY^{w \nu}\Bigr)E_\lambda^J(X_1, \dots, X_n; q,t)=\Bigl(\sum_{w \in W^J}q^{-(w^{-1}(\la),\nu)} t^{(\sigma_\la^{-1}(\rho^\vee),w(\nu))}\Bigr)E_\lambda^J(X_1, \dots, X_n; q,t).\]
	\end{cor}

	\subsection{Specializations}
	In this section we consider the specialized parasymmetric Macdonald polynomials $E_\lambda^J(X,q,0)$ and $E_\lambda^J(X,q,\infty)$
	for $\lambda \in P^-_J$.
	It was shown in \cite{OS} that for $J=\emptyset$ (i.e. in the case of nonsymmetric Macdonald polynomials) the specializations at $t=0,\infty$ are well defined. 
	
	\begin{prop}
		\label{prp::specialization}
		For any $\lambda\in P_{J}^{-}$ the specializations of $E_\lambda^J(X,q,t)$ at $t=0$ and $t=\infty$ are well defined.
	\end{prop}
	\begin{proof}
		It suffices to show that for any $\sigma\in W^J$ the polynomials $\hT_\sigma E_{w_0^J\la}(X,q,t)$ admit well-defined specializations at $t=0$ and $t=\infty$.
		Let $\la^J_+=w_0^J\la\in P_J^+$.
		It was shown in \cite{FMO1}, Proposition 2.8 and Proposition 2.9 that the desired property holds true 
		for $(t^{\frac{1}{2}})^{-l(\sigma\sigma_{\la^J_+})+l(\sigma_{\la^J_+})-l(\sigma)}\hT_\sigma E_{\la^J_+}$
		(we note that the notation of \cite{FMO1} corresponds to the notation of this paper in the following way: 
		$v_i^2 \to t$, $v_iT_i \to \hT_i$, $\sigma_\la\to\sigma_\la^{-1}$). Hence due to Corollary \ref{ShortSymmetrizer} it suffices to show that 
		\begin{equation}\label{ll}
			-l(\sigma\sigma_{\la^J_+})+l(\sigma_{\la^J_+})-l(\sigma)=0
		\end{equation}
		for any $\la^J_+\in P_J^+$
		and $\sigma\in (W^J/{\rm Stab}_{W^J} ({\lambda_+}))^{\min}$ (we note that if $\la^J_+$ is regular, then the condition on $\sigma$ simply reads as $\sigma\in W^J$).

		Recall that $\lambda_{-}$ is an antidominant weight in the $W$-orbit of $\lambda$, $\sigma_\lambda$ is the shortest element in $W$ such that $\lambda=\sigma_{\lambda}(\lambda_{-})$.
		The weight $\lambda$ is $J$-antidominant, therefore $\lambda_+^J:=w_0^{J}\lambda$ is a $J$-dominant weight in the  $W^J$-orbit of $\lambda$. The subgroup of $W^J$ that stabilizes the $J$-dominant $\lambda_+^{J}$ is generated by a subset of simple reflections (which we denote by $K$). In other words, $W^K:={\rm Stab}_{W^J}\lambda^J_+$. 
		Therefore, the minimal representative in $W^J$ that maps $\lambda$ to $\lambda_+^J$ is equal to $w_0^{K}w_0^{J}$ and its length is equal to $l(w_0^J)-l(w_0^{K})$.
		Moreover, Lemma~\ref{lem::Weyl::length} with $K=J$ and $J=I$ implies  that 
		$$
		l(\sigma_{\lambda_+^{J}}) = l(w_0^{K} w_0^{J}) + l(\sigma_{\lambda})$$
		and hence
		$\sigma_{\lambda_+^{J}} = w_0^{K} w_0^{J}  \sigma_{\lambda}$
		(since $w_0^{K} w_0^{J}  \sigma_{\lambda} \la_-=\la_+^J$ and 
		$l(w_0^{K} w_0^{J}\sigma_{\lambda})\le l(w_0^{K} w_0^{J}) + l(\sigma_{\lambda})$).
		Substituting to \eqref{ll}, we obtain that the desired equality reads as 
		\begin{equation}\label{lll}
			l(w_0^K w_0^J\sigma_\lambda) = l(\sigma w_0^K w_0^J\sigma_\lambda)+l(\sigma). 
		\end{equation}
		Lemma~\ref{lem::Weyl::length} implies that 
		$l(w_0^K w_0^J\sigma_\lambda)=l(w_0^K w_0^J) +l(\sigma_\lambda)$ and 
		$l(\sigma w_0^K w_0^J\sigma_\lambda)=l(\sigma w_0^K w_0^J) +l(\sigma_\lambda)$.
		Therefore it remains to show that
		$$l(w_0^K w_0^J) = l(\sigma w_0^K w_0^J)+l(\sigma).$$
		This is done in  Lemma~\ref{lem::subgroup}.
	\end{proof}

	\begin{lem}
		\label{lem::Weyl::length}
		Let $K\subset J$, then
		\begin{itemize}
			\item 
			$\eta\in (W^K\backslash W^J)^{\min}$ if and only if $\Phi^{K}_{+}\subset \eta(\Phi^J_{+})$;
			\item
			if $\eta\in (W^K\backslash W^J)^{\min}$, then  for any $\tau\in W^K$
			\[l(\tau\eta)=l(\tau)+l(\eta).\]
		\end{itemize}
	\end{lem}
	\begin{proof}
		If  $\Phi^{K}_{+}\subset \eta(\Phi^J_{+})$, then for any $\tau \in W^K$ we have
		\begin{equation}\label{LengthAdditionMinimal}
			\big(\tau \eta(\Phi^J_{+})\big)\cap \Phi^J_{-}=\tau\big(\eta(\Phi^J_{+})\cap \Phi^J_{-}\big) \sqcup \big( \tau(\Phi^K_{+})\cap \Phi^K_{-} \big).
		\end{equation}
		Recall that for $\tau \in W^K$ we have $l(\tau)=|\tau(\Phi^K_{+})\cap \Phi^K_{-}|.$ Therefore if $\tau$ is not the identity, then
		\[l(\tau\eta)=l(\tau)+l(\eta)>l(\eta)\]
		i. e. $\eta=\min\{W^K \eta\}$.
		
		Conversely if  $\eta\in (W^K\backslash W^J)^{\min}$, then for any simple $\alpha \in \Phi^K_+$ we have $l(s_\alpha\eta)>l(\eta)$ and therefore $\alpha \in  \eta(\Phi^J_{+})$. Thus $\Phi^{K}_{+}\subset \eta(\Phi^J_{+})$.
	\end{proof}
	
	In the following lemma for $K\subset J$ we consider the root system $\Phi^K\subset \Phi^J$, the Weyl groups $W^K \subset W^J$ and the longest elements $w_0^K$, $w_0^J$.

	\begin{lem}
		\label{lem::subgroup}
		For any 
		$\sigma \in (W^J/W^K)^{\min}$ one has $l(w_0^K w_0^J) = l(\sigma w_0^K w_0^J)+l(\sigma).$
	\end{lem}
	\begin{proof}
		Let $\tau=\sigma^{-1}$, in particular, $\tau\in (W^K\backslash W^J)^{min}$. 
		It suffices to prove that 
		there exists a reduced factorization $w_0^Kw_0^J=\tau \tau'$, i.e. $l(\tau) +l(\tau') = l(w_0^K w_0^J)$. 
		
		By construction we have
		\begin{equation}
			w_0^Kw_0^J(\Phi^J_+) \cap \Phi^J_+=\Phi^K_+. 
		\end{equation}
		Note that each element $\tau\in W^J$ is uniquely defined by the set $\tau(\Phi^J_+)\cap \Phi^J_+$.
		Using Lemma \ref{lem::Weyl::length} we have 
		\[  \tau(\Phi^J_+) \cap \Phi^J_+\supset \Phi^K_+.\]
		
		We consider two cases depending on whether there exists a  simple root $\alpha\in \Phi^J$  such that $\tau(\alpha)\in \Phi^J_+\backslash \Phi^K_+$. 
		First, assume that there is no such $\alpha$.
		Let us show that in this case $\tau=w_0^Kw_0^J$. In fact, 
		$\tau^{-1} \Phi_+^K \subset \Phi_+^J$
		(by Lemma \ref{lem::Weyl::length}) and $\tau^{-1} (\Phi^J_+\backslash \Phi^K_+)$ contains no simple roots from $\Phi_+^J$. 
		Therefore, the set  $\tau^{-1} w_0^K \Phi_+^J$ also does not contain simple roots from $\Phi_+^J$. But then  
		$\tau^{-1} w_0^K \Phi_+^J$ contains all negative simple roots from $\Phi_-^J$. Therefore 
		$\tau^{-1} w_0^K=w_0^J$.
		
		Second,  assume that there is a simple root $\alpha$ such that $\tau(\alpha)\in \Phi^J_+\backslash \Phi^K_+$. Then $\tau s_\alpha(\Phi^J_+)=\tau (\Phi^J_+)\cup \{-\tau(\alpha)\}\backslash \{\tau(\alpha)\}$. Thus 
		\[ \tau s_\alpha(\Phi^J_+) \cap \Phi^J_+\supset \Phi^K_+\]
		and $l(\tau s_\alpha)=l(\tau)+1$. Since $\tau s_\alpha\in (W^K\backslash W^J)^{min}$ and the length of $\tau s_\alpha$ is larger than the length of $\tau$, we proceed by the inverse     
		induction on $l(\tau)$ (keeping in mind the first case worked out above).    
	\end{proof}
	
	Recall the involution on $\bZ(q^{\frac{1}{2e}})[X]^{W^J}$: $X_i^\star:=X_i^{-1}$, $q^\star:=q^{-1}.$
	Let us consider the specialization of the pairing $\langle\cdot,\cdot\rangle^J$ at $t=0$.
	One has
	\[
	\mu^J_{t=0}=\prod_{k=1}^\infty \prod_{\alpha \in \Phi}(1-q^k X^\alpha)
	\prod_{\alpha \in \Phi_+}(1- X^\alpha)
	\prod_{\alpha \in \Phi_{J-}}(1- X^\alpha).
	\] 
	Then for $a,b\in \bZ(q^{\frac{1}{2e}})[X]^{W^J}$ one has
	\begin{equation}
		\label{pairing::J}
		\langle  a,b\rangle^J_{t=0}=[a b^\star ({\mu_0^J})_{t=0}]_1.
	\end{equation}
	By definition, we have:
	\begin{multline*}
		\langle f(X_1, \dots, X_n,q,t),g(X_1, \dots, X_n,q,t) \rangle^J|_{t=0}\\
		=\langle f(X_1, \dots, X_n,q,t)|_{t=0},(g(X_1, \dots, X_n,q,t)|_{t=\infty} \rangle^J_{t=0}.
	\end{multline*}
	Thus for $\la,\nu\in P_J^-$
	\begin{equation}
		\langle E_\lambda^J(X_1, \dots, X_n; q,0),E_\nu^J(X_1, \dots, X_n; q,\infty)  \rangle^J_{t=0}=0,~\text{if} ~\lambda \neq \nu
	\end{equation}
	(thanks to Definition \ref{Parasymmetric}). 
	
	We close this section by explaining the homological meaning of the 
	parabolic Cherednik pairing. Let $\fC^{J}=\fC(\cP_J)$ and ${\fC^{J}}^\vee=\fC^\vee(\cP_J)$ (see section \ref{GeneralFramework}).  
	\begin{prop}
		\label{prop::pairing::Ext}
		The $\Ext$-pairing in categories $\fC^{J}$ and ${\fC^{J}}^{\vee}$ categorifies pairing~\ref{pairing::J} in the following sense: there exists a series $c(q)$ such that 
		for any modules $M_1,N_1\in \fC^{J}$ and $M_2,N_2\in {\fC^{J}}^{\vee}$ 
		\begin{gather*}
			{\sum_{k}\sum_{i=0}^{\infty} (-1)^i q^k \dim\Ext_{\fC^{J}}^{i}(M_1,N_1\{k\}) = c(q) (\langle [M_1],[N_1] \rangle^{J}_{t=0})^{\star}},\\
			{\sum_{k}\sum_{i=0}^{\infty} (-1)^i q^k \dim\Ext_{{\fC^{J}}^\vee}^{i}(M_2,N_2\{k\}) =
				c(q)(\langle [M_2],[N_2] \rangle^{J}_{t=0})^{\star}},
		\end{gather*}
		whenever the left-hand sides are well-defined (e.g. $M_1$ and $N_2$ are finite-dimensional).
	\end{prop}
	\begin{proof}
		The proof is the direct corollary from  Theorem~\ref{thm::Ext::pairing} and formula ~\eqref{pairing::J}. Note that thanks to equation~\eqref{eq::kernel::P_J} the kernel $\mu_{\cP_J}$ is proportional to the kernel $\mu^{J}_{t=0}$.
		In particular, the constant $c(q)=[\mu_{t=0}^{J}]_{1} \prod_{k=1}^{\infty}(1-q^k)^{-n}.$
	\end{proof}

	\section{Parahoric algebras and stratified categories: definitions}
	\label{ParaCat}
	In this section we recall the formalism of the stratified categories and specialize the general framework of section \ref{GeneralFramework} to the case $\fa=\cP_J$.

	\subsection{Recollection on Stratified categories}
	\label{RSC}
	In this subsection, we recall the notion of a {stratified category} that generalizes the notion of a {highest weight category}; for more details, see our previous paper~\cite{FKhMO} and \cite{Br,BWW,BS,CPS,Kh}.
	
	Let $\Theta$ be a partially ordered set such that for any $\lambda\in\Theta$ there are finitely many elements smaller than $\lambda$. Let $\cO$ be an abelian category
	with simple objects labeled by a set $\Lambda$. 
	
	\begin{dfn}\label{Ordered:Category}
		We call an abelian category  $\cO$ $\Theta$-ordered if 
		\begin{itemize}
			\item $\cO$ enjoys the Krull-Schmidt property,
			\item the Grothendieck group $K_0(\cO)$ is well defined,
			\item $\cO$ is equipped with the map $\rho:\Lambda\to \Theta$.
		\end{itemize}
	\end{dfn}
	We call the partially ordered set $\Theta$ the set of {weights}. 
	
	For each $\lambda\in\Lambda$ we define the Karoubi envelope $\cO_{\leq\lambda}$ in $\cO$ of the set of irreducibles $L_\mu$ with $\rho(\mu)\leq\rho(\lambda)$.
	The fully faithful embedding $\imath_{\lambda}:\cO_{\leq\lambda} \rightarrow \cO$ admits a left adjoint functor $\imath_{\lambda}^*$ and the right adjoint functor $\imath_{\lambda}^!$:
	\begin{equation}\label{diagram}
		\begin{tikzcd}
			\cO_{\leq\lambda}
			\arrow[rrr,"\imath_{\lambda}"description,"\perp" near start,"\perp"' near start]
			&&&
			\cO
			\arrow[lll,shift left=2ex,"\imath_{\lambda}^!(M):=M_{\leq\lambda}"]
			\arrow[lll,shift right=2ex,"\imath_{\lambda}^*(M):=M/((M)_{\not\leq\lambda})"']
		\end{tikzcd},
	\end{equation}
	where $\imath_{\lambda}^{!}M$ is the maximal submodule of $M$ that belongs to $\cO_{\leq\lambda}$ and the submodule $(M)_{\not\leq\lambda}\subset M$ is generated by all weights that are not less or equal than $\rho(\lambda)$. 
	In particular, $\imath_{\lambda}^*(M)$ is the maximal quotient of $M$ that belongs to the subcategory $\cO_{\leq\lambda}$.
	
	Similarly, for any finite subset $S\subset\Lambda$ there exists a subcategory $\cO_{\leq S}$, a fully faithful functor $\imath_{\leq S}:\cO_{\leq S}\to \cO$, and its adjoints $\imath_{\leq S}^{*},\imath_{\leq S}^{!}$,
	such that for each simple objects $L_\mu$ showing up in the Krull-Schmidt decomposition of $M\in\cO_{\leq S}$ there exists $s\in S$ such that $\rho(\mu)\leq\rho(s)$.
	
	\begin{dfn}\label{StandardCostandard}
		Let $\la\in\Lambda$.
		\begin{itemize}
			\item 
			\emph{The standard module} $\Delta_{\lambda}\in \cO_{\leq\la}$ is the projective cover of the irreducible module $L_\lambda$.
			\item 
			\emph{The costandard module} $\nabla_{\lambda}\in \cO_{\leq\la}$ is the injective hull of $L_{\lambda}$;
			\item
			\emph{The proper standard module} $\overline{\Delta}_{\lambda}$ is an indecomposable object of $\cO_{\leq\la}$ that satisfies the universal properties of the projective covers in the exact subcategory 
			$$\overline{\cO}_{\leq\lambda}:=\{M\in\cO_{\leq\lambda} \colon [M:L_{\lambda}]\leq 1\} $$ 
			consisting of modules such that the multiplicity of $L_{\lambda}$ is at most $1$.
			\item 
			\emph{The proper costandard module} $\overline{\nabla}_{\lambda}$ is the injective hull of $L_{\lambda}$ in $\overline\cO_{\leq\lambda}$.
		\end{itemize}
	\end{dfn}
	
	We summarize several known equivalent definitions of the {stratified category} in  Definition~\ref{dfn::stratified}. For the proof of the equivalences between assumptions $(s1)$, $(s2)$ and $(s3)$ see e.g. our previous paper~\cite{FKhMO} \S 1.
	
	\begin{dfn}
		\label{dfn::stratified}
		Consider a $\Theta$-ordered category $\cO$ with enough projectives and its derived category $\cD(\cO)$
		generated by projectives. Assume that proper costandard objects $\overline\nabla_{\lambda}$ exist in $\cO$ and 
		the standard object $\Delta_{\lambda}$  is a projective cover of $L_\lambda$ in all categories $\cO_{\le S}$  such that $\la\in S$ and $\rho(\lambda)$ is the maximal element in $\rho(S)$. Then $\cO$ is called stratified if  
		one of the following equivalent conditions is satisfied:
		\begin{enumerate}
			\item[(s1$^\vee$)] 
			for any finite subset $S\subset \Lambda$ the derived functor $\imath_{\leq S}: \cD(\cO_{\leq S}) \rightarrow \cD(\cO)$ is a fully faithful embedding;
			\label{item::D} 
			\item[(s2$^\vee$)] $\forall \lambda\in\Lambda$ the kernel of the projection $\imath_{\lambda}^{*}:\bP_\lambda\rightarrow \Delta_{\lambda}$ admits a filtration by $\Delta(\mu)$ with $\rho(\mu)>\rho(\lambda)$;
			\label{item::P}
			\item[(s3$^\vee$)] $\RHom_{\cD(\cO)}(\Delta_{\lambda},\overline\nabla_{\mu}) = 0$ for all pairs of $\lambda,\mu$ such that $\rho(\lambda)\neq\rho(\mu)$, i.e. there are no higher derived homomorphisms.
			\label{item::Ext}
		\end{enumerate}.

		Similarly, suppose that $\Theta$-ordered category $\cO$ has enough injectives, $\cD(\cO)$ is generated by injectives. Assume further that proper standard modules $\overline\Delta_\lambda$ exist in $\cO$, and  $\nabla_{\lambda}$ remains being  injective hull of $L_\lambda$ in subcategories $\cO_{\leq S}$ such that 
		$\la\in S$ and $\rho(\lambda)$ is the maximal element of $\rho(S)$. Then we say that $\cO$ is stratified if one of the following equivalent assumptions is satisfied:
		\begin{enumerate}
			\item[(s1)] 
			for any finite subset $S\subset \Lambda$ the derived functor $\imath_{\leq S}: \cD(\cO_{\leq S}) \rightarrow \cD(\cO)$ is a fully faithful embedding;
			\item[(s2)] $\forall \lambda\in\Lambda$ the cokernel of the injection $\imath_{\lambda}:\nabla_{\lambda}\rightarrow \bI_{\lambda}$ admits a filtration by $\nabla_\mu$ with $\rho(\mu)>\rho(\lambda)$;
			\item[(s3)] $\forall \lambda,\mu$ with $\rho(\lambda)\neq\rho(\mu)$ we have  $ \Ext_{\cD(\cO)}^{\udot}(\overline\Delta_\lambda,\nabla_\mu) = 0$.
		\end{enumerate}    
	\end{dfn}

	\subsection{Degree shifts}\label{gradedshift}
	Simple objects in categories $\fC(\cP_J)$ and ${\fC^\vee(\cP_J)}$ are labeled by pairs $(\la,k)$, but the 
	Cherednik order and the dual Cherednik order have nothing to do with the 
	parameter $k$ (see below for the details). Let us describe this situation in more detail.

	Suppose that the abelian category $\cO$ we are dealing with admits a grading, meaning that 
	each object is graded and 
	there is an autoequivalence of shifting the grading (i.e. the $m$-th graded component of the module $M\{k\}$ is the $(n-k)$-th graded component of the module $M$).
	We also assume that all homomorphisms $\Hom_{\cO}(M,N)$, as well as all extension groups, preserve the grading. In this situation it is reasonable to deal with the collections of  graded shifts:
	$$
	q\ttt\Hom_{\cO}(M,N):= \oplus_{k\in\bZ} \Hom_{\cO}(M\{k\},N);
	\quad
	q\ttt\Ext_{\cO}^{p}(M,N):= \oplus_{k\in\bZ} \Ext_{\cO}^{p}(M\{k\},N).
	$$
	Suppose also that each simple object in $\cO$ has only one nontrivial homogeneous component,  the simple objects in $\cO$ are labeled by $\Lambda\times \bZ$ for appropriate set $\Lambda$ and 
	the graded shift maps simples to simples: $ L_{\lambda,k}\{k\} \simeq L_{\lambda,0}$.
	Then the projective cover (resp., the injective hull) of a shifted module $M\{k\}$ is isomorphic to the graded shift of the projective cover (resp., the injective hull) of $M$.
	
	Suppose that the projection $\rho:\Lambda\times\bZ\to \Theta$ factors through the projection $\Lambda\times\bZ\to \Lambda$. Then the shift autoequivalence maps a (proper) (co)standard module to a shifted (proper) (co)standard module:
	$$
	\Delta_{\lambda,k}\{k\} = \Delta_{\lambda,0}, \quad \overline{\Delta}_{\lambda,k}\{k\}  = \overline{\Delta}_{\lambda,0};
	\quad
	\nabla_{\lambda,k}\{k\} = \nabla_{\lambda,0}, 
	\quad
	\overline{\nabla}_{\lambda,k}\{k\}  = \overline{\nabla}_{\lambda,0}.
	$$
	In what follows we omit the second grading index in the notation of the (co)standard objects:
	$$
	\Delta_\lambda:= \Delta_{\lambda,0},  
	\quad
	\overline{\Delta}_{\lambda}:= \overline{\Delta}_{\lambda,0}, 
	\quad
	\nabla_\lambda:= \nabla_{\lambda,0}, 
	\quad 
	\overline{\nabla}_{\lambda}:= \overline{\nabla}_{\lambda,0}.
	$$
	We also combine all the extension groups together with all their shifts. In particular, the $\Theta$-ordered category $\cO$ is stratified if for all $\lambda,\mu\in\Lambda$ such that $\rho(\lambda)\neq\rho(\mu)$ we have:
	\begin{multline*}
		\Ext_{\cO}(\overline{\Delta}_\lambda,\nabla_{\mu})=0 \stackrel{def}{\quad \Leftrightarrow \quad} q\ttt\Ext_{\cO}(\overline{\Delta}_{\lambda,0},\nabla_{\mu,0}) = 0
		\quad \Leftrightarrow \quad
		\\
		\Leftrightarrow \forall k \ 
		\Ext_{\cO}(\overline{\Delta}_{\lambda,k},\nabla_{\mu,0}) = 0
		\quad \Leftrightarrow \quad
		\forall k,m \ 
		\Ext_{\cO}(\overline{\Delta}_{\lambda,k},\nabla_{\mu,m}) = 0.
	\end{multline*}

	\subsection{Category of modules over parahoric Lie algebra}
	\label{sec::parahoric}
	In this section, we consider the abelian categories $\fC^{J}$ and $\fC^{J^\vee}$ of graded $\cP_J$-modules whose simple objects are indexed by $J$-antidominant weights and the grading parameter.
	In Section~\ref{Main} we show that these categories are stratified in the sense of Section~\ref{RSC} for the Cherednik partial order and the dual partial order.
	
	Let us briefly recall the notation from Section \ref{sec::parabolic::setup}.
	We fix a subset $J\subset \{1,\dots,n\}$ and denote by  
	$\fl_J\subset \fp_J$ the corresponding Levi and parabolic subalgebras
	(in particular, $\fb\subset \fp_J\subset \fg$).
	Let $\cP_J\supset \fp_J$ be the standard parahoric Lie algebra (an affine analog of the parabolic subalgebra)  
	$\mathcal{P}_J = \fp_J \oplus z\fg[z].$
	Then $\mathcal{P}_J= \fl_J\oplus \fr_J$ for the radical $\fr_J$.
	One has two extreme cases: if $J=\varnothing$, then $\mathcal{P}_J$ is the Iwahori algebra $\cI$; if 
	$J=\{1,\dots,n\}$, then $\mathcal{P}_J$ coincides with the current algebra $\fg[z]$.
	In general, $\cI\subset \mathcal{P}_J\subset \fg[z]$ and 
	$\mathcal{P}_I\subset \mathcal{P}_J$ if and only if $I\subset J$.

	Recall the  categories $\fC^{J}$ and ${\fC^{J}}^{\vee}$ of $\cP_J$ modules.
	The simple objects in $\fC^J$ are 
	of the form  $L_{\la,k}$, $\la\in P_J^-$, $k\in \bZ$. Each $L_{\la,k}$ is isomorphic to an irreducible $\fl_J$ module $L_\la$, on which the  radical $\fr_J$  acts trivially; 
	the module $L_\la$ is placed in degree $k$.
	The objects of $\fC^J$ are the $\mathcal{P}_J$-modules $M$ subject to the following conditions: 
	\begin{itemize}
		\item $M=\bigoplus_{k\in\bZ} M_k$, $M_k=0$ for $k$  large enough and $xz^a M_k\subset M_{k+a}$ for any $xz^a\in \mathcal{P}_J$;
		\item each $M_k$ admits a  decomposition 
		$M_{k}=\bigoplus_{\nu\in P_J^-} L_{\nu,k}$, where each $L_{\nu,k}$ shows up finitely many times;
		\item $M$ is cogenerated by a subspace $\overline{M}$ admitting $\fh$ decomposition  $\overline{M}=\bigoplus_{\nu\in P} \overline{M}(\nu)$ such that $\dim \overline{M}(\nu)<\infty$ for all $\nu\in P$.
	\end{itemize}

	We also have a similar definition for the (dual) category ${\fC^J}^\vee$, where the last property above is replaced by 
	\begin{itemize}
		\item $M$ is generated by a subspace $\overline{M}$ admitting an $\fh$-weight decomposition $\overline{M}=\bigoplus_{\nu\in P} \overline{M}(\nu)$ such that $\dim \overline{M}(\nu)<\infty$ for all $\nu\in P$.
	\end{itemize}
	We denote simple objects of ${\fC^J}^\vee$ by the same symbols $L_{\la,k}$.
	
	For $J=\emptyset$ the categories $\fC^{\varnothing}$ and ${\fC^{\varnothing}}^{\vee}$ coincide with the categories $\fC$ and $\fC^{\vee}$ discussed in our previous paper~\cite{FKhMO}. 
	If it does not produce any ambiguity we omit the empty set index $\varnothing$ in 
	the notation $\fC^{\varnothing}$ and ${\fC^{\varnothing}}^{\vee}$. 
	
	\begin{prop::def}
		The category $\fC^{J}$ has enough injectives and each module $M\in\fC^J$ admits an injective resolution $Q^{\udot}$ such that for each irreducible $L_{\lambda,k}\in\fC^{J}$ the total multiplicity of $L_{\lambda,k}$ 
		in the Jordan-Holder series of $Q^k$ is finite. 
		
		Similarly, the category ${\fC^{J}}^{\vee}$ has enough projectives and each module $M\in {\fC^{J}}^{\vee}$ admits a projective resolution with finite multiplicities of each irreducible module.     
	\end{prop::def}
	\begin{proof}
		The injective hull of the irreducible module $L_{\lambda,k}$ in the category $\fC^{J}$ is given by
		\[ \mathbb{I}_\lambda=\mathrm{coInd}_{\fl_J}^{\cP_J}L_{\lambda,k}=L_{\lambda,k}\otimes \U(\fr_J)^\vee.\]
		The projective cover of the irreducible module $L_{\lambda,k}$ in the category ${\fC^{J}}^{\vee}$ is given by
		\[\mathbb{P}_\lambda = \mathrm{Ind}_{\fl_J}^{\cP_J}L_{\lambda,k}=L_{\lambda,k}\otimes \U(\fr_J).\]
		Note that each irreducible module $L_{\lambda,k}$ admits a projective/injective resolution thanks to the parabolic BGG resolution (see e.g.~\cite{Kum}).
	\end{proof}
	
	In what follows we consider bounded from below derived category $\bD_+(\fC^{J})$ and bounded from above derived category is denoted by $\bD_-({\fC^{J}}^{\vee})$.
	The graded embedding of the Iwahori Lie subalgebra $\cI\hookrightarrow \cP_{J}$ induces a pair of exact restriction functors that do nothing with the underlying vector spaces:
	\begin{equation}
		\label{eq::res::functors}
		\Res_{J}:\fC^{J} \to \fC^{\varnothing} = \fC, \quad 
		\Res_{J}: {\fC^{J}}^{\vee} \to {\fC^{\varnothing}}^{\vee} = \fC^{\vee}.  
	\end{equation}
	\begin{prop::def}
		\begin{itemize}
			\item
			The restriction functor $\Res_{J}:\fC^{J} \to \fC$ admits a right adjoint $\Coind_{J}:\fC \to \fC^{J}$ that maps a module $M\in\fC$ to the maximal $\fp_{J}$-integrable submodule of the coinduced module 
			$$\Hom_{U(\cI)}(U(\cP_J),M) \simeq \Hom_{U(\fb)}(U(\fp_{J}),M);$$  
			\item 
			the right adjoint functor $\Coind_{J}$ is left exact and admits a right derived functor 
			$$\sR\Coind_{J}: \bD_{+}(\fC) \rightarrow \bD_{+}(\fC^{J}).
			$$
			\item 
			The restriction functor $\Res_{J}:{\fC^{J}}^{\vee} \to \fC^{\vee}$ admits a left adjoint $\Ind_{J}:\fC^{\vee} \to {\fC^{I}}^{\vee}$ that maps an $\cI$-module $M\in\fC$ to the maximal $\fp_{J}$-integrable quotient of the induced module 
			$$
			\U(\cP_{J})\otimes_{\U(\cI)} M \simeq 
			\U(\fp_{J})\otimes_{\U(\fb)} M; 
			$$
			\item 
			the left adjoint functor $\Ind_{J}$ is right exact and admits a left derived functor 
			$$\sL\Ind_{J}: \bD_{-}(\fC^{\vee}) \rightarrow \bD_{-}({\fC^{J}}^{\vee}).$$
		\end{itemize}
	\end{prop::def}
	\begin{proof}
		It is well known that for a subalgebra $A\hookrightarrow B$ the restriction functor $\Res:B\ttt mod \rightarrow A\ttt mod$ is an exact functor whose left adjoint is the induction functor $\Ind: M \mapsto B\otimes_{A} M$ and right adjoint is the conduction functor $\Coind: M\mapsto \Hom_{A}(B,M)$ (see \cite{Weibel}). 
		For $B=\U(\cP_J)$ and $A=\U(\cI)$ we get the functors between categories $\fC^{\vee}$ and ${\fC^J}^{\vee}$, 
		modulo the assumption that a module $B\otimes_A M$ may not belong to ${\fC^J}^{\vee}$ and one has to take its maximal quotient that belongs to this category;
		we end up with the description of the left adjoint functor $\Ind_J$. 
		The same arguments explain the detailed description of the right adjoint functor $\Coind_J$. The existence of their derived functors follows from the general description of derived functors and the existence of enough projectives in ${\fC^J}^{\vee}$ and enough injectives in $\fC^J$ correspondingly.
	\end{proof}
	
	The set of simple objects $L_{\lambda,k}$ in categories $\fC^{J}$ and $\fC^{J^{\vee}}$ are indexed by pairs $(\lambda,k)$, $\lambda\in P_J^{-}$, $k\in\bZ$.
	The set of $J$-antidominant weight admits two natural partial orders: the Cherednik partial $\preceq$ order and the dual Cherednik partial order $\preceq^{\vee}$.
	Therefore, we can apply the formalism of $\Theta$-ordered categories from Definition~\ref{Ordered:Category} and define two different filtrations on the abelian categories $\fC^{J}$ and ${\fC^{J}}^{\vee}$ (considering a projection $P_J^{-}\times\bZ\twoheadrightarrow P_J^{-}$). 
	The Serre subcategories generated by irreducibles $L_{\mu,k}$ with $\mu\preceq\lambda$ or $\mu\preceq^\vee\lambda$ are denoted by $\fC^{J}_{\preceq\lambda}$,  $\fC^{J^{\vee}}_{\preceq\lambda}$ and $\fC^{J}_{\preceq^\vee\lambda}$,  $\fC^{J^{\vee}}_{\preceq^\vee\lambda}$. 
	We denote by $\Delta_{\lambda,k}$ and $\nabla_{\lambda,k}$ the standard module in $\fC^{J^\vee}$ and the costandard module in $\fC^{J}$. In what follows we usually omit the second index, since it does not affect the corresponding modules (see subsection \ref{gradedshift} for more details).

	\section{Standard modules for parahoric algebras}
	\label{sec::standards}
	In this section, we provide an explicit construction of standard and costandard modules over parahoric Lie algebras $\cP_J$. 
	First we define two families of {global} and {local} modules $\mathbb{D}_\lambda^{J}$ and $D_\lambda^J$, $\mathbb{U}_{\lambda}^{J}$ and $U_{\lambda}^{J}$ generalizing the analogous collection  of modules
	$\mathbb{D}_\lambda$, $D_\lambda$, $\mathbb{U}_{\lambda}$ and $U_{\lambda}$ over Iwahori algebra
	(see \cite{FKM,FKhMO}).
	Then we give a detailed description of these modules considered as modules over the Iwahori algebra. 
	The results of this Section are used in the next one
	for the proofs of the main results.  
	\subsection{Definitions}
	
	Recall that  
	any $\lambda \in P$ defines an antidominant weight $\la_-\in P^-$ and $\sigma_\lambda\in W$ such that
	$\sigma_\lambda(\lambda_-)=\la$ and  $\sigma_\lambda$ is the shortest element with this property. 
	
	\begin{dfn}
		\label{Ddefrel}
		The module $\mathbb D_\lambda^J$ is the cyclic $\mathcal{P}_J$-module generated by the element $v_{\lambda}$ of weight 
		$\lambda\in P_J^-$ 
		defined by the following relations:
		\begin{equation}
			\begin{cases}\label{StandardParahoricRelations}
				\widehat{\sigma_\lambda}(e_\alpha z^k)v_{\lambda}=0,~\text{if}~ \alpha \in \Phi_-, k=1, \dots\\
				e_\alpha v_{\lambda}=0,~\text{if}~ \alpha \in \Phi_{J-}\\
				\widehat{\sigma_\lambda}(e_\alpha)^{-\langle \alpha^\vee,\lambda_- \rangle+1}v_{\lambda}=0,~\text{if}~\alpha \in \Phi_+, \sigma_\lambda(\alpha) \in \Phi_+
				\\
				\widehat{\sigma_\lambda}(e_\alpha)^{-\langle \alpha^\vee,\lambda_- \rangle}v_{\lambda}=0,~\text{if}~\alpha \in \Phi_+, \sigma_\lambda(\alpha) \in \Phi_-.
			\end{cases}      
		\end{equation}
		The module $D_\lambda^J$ is defined by  relations \eqref{StandardParahoricRelations} and the relations
		\begin{equation}\label{localD}
			h z^k v_{\lambda}=0, h \in \fh, k=1,\dots.
		\end{equation}
	\end{dfn}
	
	In a similar way we define another family of modules.
	
	\begin{dfn}
		\label{Udefrel}
		The module $\mathbb U_\lambda^J$ is the cyclic $\mathcal{P}_J$-module generated by the element $v_{\lambda}$ of weight 
		$\lambda\in P_J^-$ defined by the relations:
		\begin{equation}
			\begin{cases}\label{CoStandardParahoricRelations}
				\widehat{\sigma_\lambda}(e_\alpha z^k)v_{\lambda}=0,~\text{if}~ \alpha \in \Phi_-, k=1, \dots\\
				e_\alpha v_{\lambda}=0,~\text{if}~ \alpha \in \Phi_{J-},\\
				\widehat{\sigma_\lambda}(e_\alpha)^{-\langle \alpha^\vee,\lambda_- \rangle+1}v_{\lambda}=0,~\text{if}~\alpha \in \Phi_+, \sigma_\lambda(\alpha) \in \Phi_-\cup \Phi_J\\
				\widehat{\sigma_\lambda}(e_\alpha)^{-\langle \alpha^\vee,\lambda_- \rangle}v_{\lambda}=0,~\text{if}~\alpha \in \Phi_+, \sigma_\lambda(\alpha) \in \Phi_+\cap (\Phi \backslash \Phi_J).
			\end{cases}      
		\end{equation}
		The module $U_\lambda^J$ is defined by relations \eqref{CoStandardParahoricRelations} and relations \eqref{localD}.
	\end{dfn}
	
	In what follows we identify the restrictions of the $\mathcal{P}_J$ modules  
	defined above to $\cI$  with  Weyl modules with characteristics $W_\la(r)$ and $\mathbb W_\la(r)$ (see \cite{FeMa,FMO1}).
	These cyclic $\cI$ modules depend on an element of the affine Weyl group, its reduced decomposition and a positive integer. The generalized Weyl modules with characteristics   are defined by explicit sets of relations. 
	In what follows we use the special cases of these modules, see the proofs of Lemma \ref{DRestriction} and of Proposition \ref{URestriction}.

	\subsection{Restriction to the Iwahori algebra: the modules D}
	In this subsection, we describe the restrictions to the Iwahori algebra of the modules 
	$\mathbb D_\lambda^J$ and $D_\lambda^J$. For a $\cP_J$ module $M$ we denote by
	$\Res_{J} M$ the restriction of $M$ to the Iwahori algebra $\cI\subset \cP_J$.

	\begin{lem}\label{DRestriction}
		The modules
		$\Res_{J}\mathbb D_\lambda^J, \Res_{J}D_\lambda^J$
		are cyclic $\cI$ modules with the defining relations coinciding with the defining relations of the initial modules
		$\mathbb D_\lambda^J, D_\lambda^J$
		with the second line  in \eqref{StandardParahoricRelations} omited.  
	\end{lem}
	\begin{proof}
		It suffices to show that the action of the Iwahori algebra on the  $\cI$ modules defined by relations 
		\eqref{StandardParahoricRelations} and \eqref{CoStandardParahoricRelations} with second lines omitted (and with the relations \eqref{localD} added in the case of local modules) can be extended to the action of the parahoric algebra $\cP_J$ (by adding the conditions from the second lines of \eqref{StandardParahoricRelations}).
		
		Let us first show that it suffices to work out either global or local case (i.e. the two cases are equivalent). 
		In fact, the statement we need to show can be restated as follows.  Let us take a defining relation $Rv_\la=0$ of one of the $\cI$ modules $M$ we are interested in. 
		Then for  any $\al\in\Phi_{J-}$ one has $[e_\alpha,R]v_\la=0$ assuming that $e_\beta v_\la=0$ for any $\beta\in\Phi_{J-}$. 
		Now the commutation relations 
		$[e_\al,h_\beta z^k]={\rm const}.e_\al z^k$ imply the equivalence between the global and local cases.
		
		Let us prove the desired statement for the (local) modules $\Res_{J}D_\lambda^J$. 
		We consider the local Weyl module $W_{\la_-}$ (recall $\la=\sigma_\la(\la_-))$. Let 
		$u_\la\in W_{\la_-}$ be a weight $\la$ extremal vector. Since $\la\in P_J^-$, one has 
		$\U(\cI)u_\la = \U(\cP_J)u_\la$. So it suffices to show that the $\cI$-module $\U(\cI)u_\la\subset W_{\la_-}$ is defined by relations \eqref{localD} and relations  \eqref{StandardParahoricRelations} with the second line omitted. 
		To this end, we construct the decomposition of the generalized Weyl module $W_{\la_-}$ (see (\cite{FeMa,FMO1})) such that
		\begin{itemize}
			\item     
			the module $\U(\cI)u_\la $ coincides with a generalized Weyl module with characteristics $M$, which shows up as a submodule in the decomposition as above;
			\item 
			the defining relations of $M$  are \eqref{localD} and  \eqref{StandardParahoricRelations} with the second line omitted.  
		\end{itemize}
		
		We consider the translation elements ${\mathbf t}_{\lambda_-}, {\mathbf t}_{\lambda_-}\in W^e$
		and the decomposition ${\mathbf t}_\lambda = r_\lambda \sigma_\lambda^{-1}$ for $\lambda\in P$. Conjugation, we get the reduced decomposition ${\mathbf t}_{\lambda_-}=\sigma_\la^{-1} r_\la$, i.e., $\ell({\mathbf t}_{\la_-})=\ell(\sigma_\la^{-1})+\ell(r_\la)$.
		Having hand-reduced expressions
		\[
		\sigma_\la^{-1}=s_{i_1}\dots s_{i_r},\ r_\la=\pi s_{i_{r+1}}\dots s_{i_M},
		\]
		we obtain a reduced expression
		\begin{equation}\label{reddecV}
			{\mathbf t}_{\la_-}=\pi s_{\pi^{-1}i_1}\dots s_{\pi^{-1} i_r} s_{i_{r+1}}\dots s_{i_M}.
		\end{equation}
		Hence, with these choices of reduced expressions, we have 
		for $j=1,\dots,r$
		\begin{multline*}
			\beta^\vee_{j}({\mathbf t}_{\la_-})= s_{i_M}\dots s_{i_{r+1}}s_{\pi^{-1} i_r}\dots s_{\pi^{-1} i_{j+1}}(\pi^{-1}\alpha_{i_j}^\vee)={\mathbf t}_{\la_-}^{-1}\pi
			s_{\pi^{-1}i_1}\dots s_{\pi^{-1}i_j}(\pi^{-1}\alpha_{i_j}^\vee).
		\end{multline*}
		Since the translation element ${\mathbf t}_{\la_-}^{-1}$ does not affect the finite parts of the roots $\beta_\bullet$, one gets
		\begin{equation}\label{tbegining}
			\overline\beta^\vee_{j}({\mathbf t}_{\la_-})=s_{i_1}\dots s_{i_{j-1}}(-\alpha_{i_j}^\vee).
		\end{equation}
		
		Let us consider the decomposition of $W_{\lambda_-}$ with respect to the sequence 
		\[\overline\beta^\vee_{1}({\mathbf t}_{\la_-}),\dots,\overline\beta^\vee_{r}({\mathbf t}_{\la_-}).\]
		Recall that in a nutshell a step of the decomposition procedure of \cite{FeMa} consists of taking a cyclic submodule of a cyclic $\cI$ module such that both the submodule and the corresponding quotient are identified with 
		generalized Weyl modules with characteristics. 
		The subquotients of the decomposition procedure correspond to the paths in the quantum Bruhat graph (see \cite{OS,FeMa}), enumerated by subsequences of $\overline\beta$'s. In particular, 
		the path corresponding  to the full sequence 
		$\overline\beta^\vee_1({\mathbf t}_{\la_-}), \dots,\overline\beta^\vee_r({\mathbf t}_{\la_-})$ 
		corresponds to a certain submodule as soon as one checks that the elements 
		$s_{\overline\beta^\vee_{1}({\mathbf t}_{\la_-})}\dots s_{\overline\beta^\vee_{j}({\mathbf t}_{\la_-})}$, $j=1,\dots,r$ indeed form a path in the quantum Bruhat graph.
		Let us check this statement.
		
		One has 
		\[s_{\overline\beta^\vee_{1}({\mathbf t}_{\la_-})}\dots s_{\overline\beta^\vee_{j}({\mathbf t}_{\la_-})}=s_{i_j}\dots s_{i_1}\in W\]
		and
		\begin{equation}\label{l+1}l(s_{\overline\beta^\vee_{1}({\mathbf t}_{\la_-})}\dots s_{\overline\beta^\vee_{j+1}({\mathbf t}_{\la_-})})=l(s_{\overline\beta^\vee_{1}({\mathbf t}_{\la_-})}\dots s_{\overline\beta^\vee_{j}({\mathbf t}_{\la_-})})+1.
		\end{equation}
		One also has
		\[s_{\overline\beta^\vee_{1}({\mathbf t}_{\la_-})}\dots s_{\overline\beta^\vee_{r}({\mathbf t}_{\la_-})}=\sigma_\la.\]
		So we do have a path in the Bruhat graph (and hence in the quantum Bruhat graph as well). The corresponding path ends at $\sigma_\la$ and hence one gets a cyclic submodule of $W_{\la_-}$ generated from the weight $\la$ extremal vector of $W_{\la_-}$ as a submodule in our decomposition procedure. 
		Indeed, the Weyl module $W_{\la_-}$ contains an irreducible $\fg$ subrepresentation $V_{\la_-}$. 
		Hence (due to \eqref{l+1}) all the cyclic vectors of the submodules of $W_{\la_-}$ in the decomposition procedure are generated (as $\cI$-modules) by the extremal vectors of $V_{\la_-}\subset W_{\la_-}$.   
		
		By definition, the submodule in the decomposition as above is 
		$W_\la(r)$. The defining relations of this module are given by (see Proposition 4.14, \cite{FKhMO})
		\begin{align*} \nonumber
			\widehat{\sigma_\lambda}(e_\alpha z^k)v_{\lambda}=0,~\text{if}~ \alpha \in \Phi_-, k=1, \dots\\
			\widehat{\sigma_\lambda}(e_\alpha)^{-\langle \alpha^\vee,\lambda_- \rangle+1}v_{\lambda}=0,~\text{if}~\alpha \in \Phi_+, \sigma_\lambda(\alpha) \in \Phi_+,
			\\ \nonumber \widehat{\sigma_\lambda}(e_\alpha)^{-\langle \alpha^\vee,\lambda_- \rangle}v_{\lambda}=0,~\text{if}~\alpha \in \Phi_+, \sigma_\lambda(\alpha) \in \Phi_-,\\ \nonumber
			h z^k v_{\lambda}=0, h \in \fh, k=1,\dots.
		\end{align*}
		These are exactly the expected relations for  $\Res_{J}D_\lambda^J$.
	\end{proof}

	\begin{cor}\label{StandardRestriction}
		\[\Res_{J}\mathbb D_\lambda^J\simeq \mathbb D_\lambda^\varnothing,\quad \Res_{J}D_\lambda^J\simeq D_\lambda^\varnothing.\]  
	\end{cor}
	\begin{proof}
		Follows from  Lemma~\ref{DRestriction} and Proposition 4.14, \cite{FKhMO}.
	\end{proof}
	
	\begin{cor}
		The $\cI$ submodule of $W_{\lambda_-}$, $\lambda_-\in P^-$ generated by a vector of the extremal weight $\lambda=\kappa(\lambda_-)$, $\kappa \in W$, is isomorphic to $D_{\lambda}^{\varnothing}$.
	\end{cor}
	
	\begin{lem}\label{GlobalWeyl}
		Let $\la_-\in P^-$ and $\la=\sigma_\la(\la_-)$. Then one has a natural embedding of generalized global Weyl modules $\W_\la\subset \W_{\la_-}$.
		If $\la\in P_J^-$, then 
		the natural $\cI$ action on $\W_\la$ 
		admits a natural extension to the action of the parahoric algebra $\cP_J$. 
	\end{lem}
	\begin{proof}
		It suffices to show that   $\W_\la$ sits inside $\W_{\la_-}$ as an $\cI$ submodule generated by the weight $\la$ (and $z$-degree $0$) extremal vector $u_\la$. 
		This would imply the desired claim since  $\la\in P_J^-$ and hence $\U(\cI)u_\la=\U(\cP_J)u_\la$ (note that $\cP_J$ acts on
		$\W_{\la_-}$ since this module is acted upon the current algebra $\fg[t]$).
		
		We consider the same decomposition procedure as in the proof of Lemma \ref{DRestriction}. 
		Recall that in the proof of Lemma  \ref{DRestriction} we constructed the set of coroots such that the corresponding path in the Bruhat graph corresponds to a cyclic submodule inside $W_{\la_-}$. 
		Now let us take the same set of roots and use it for the decomposition procedure of the global Weyl module 
		$\W_{\la_-}$ (see \cite{FMO1}). Then the corresponding cyclic $\cI$ submodule (generated by the extremal weight $\la$ vector $u_\la$) is isomorphic to the global Weyl module $\W_\la$.
	\end{proof}
	
	\subsection{Restriction to the Iwahori algebra: the modules U}
	
	In this subsection, we describe the restrictions to the Iwahori algebra of the modules
	$\mathbb U_\lambda^J$ and $U_\lambda^J$.

	\begin{lem}
		The restrictions 
		$\Res_{J}\mathbb U_\lambda^J, \Res_{J}U_\lambda^J$
		are cyclic $\cI$ modules with the defining relations coinciding with the defining relations of the initial modules
		$\mathbb U_\lambda^J, U_\lambda^J$
		with the second line  in \eqref{CoStandardParahoricRelations} omited.  
	\end{lem}
	\begin{proof}
		As in Lemma \ref{DRestriction} the equality $[e_\al,h_\beta z^k]={\rm const}.e_\al z^k$ implies that the questions about local and global modules are equivalent. We prove the statement for the global modules $\mathbb U_\lambda^J$.
		
		Let $\overline{\Res}_J\mathbb U_\lambda^J$ be the cyclic $\cI$-module defined by relations of
		$\mathbb U_\lambda^J$
		with the second line  in \eqref{CoStandardParahoricRelations} omitted.
		Let us show the following isomorphism of $\cI$-modules:
		\begin{equation}\label{Uint}
			\overline{\Res}_J\mathbb U_\lambda^J\simeq \mathbb W_\la \left/ 
			\sum_{\substack{W\ni \kappa> \sigma_\la \\ \kappa\la_-\in P_J^-}} \mathbb W_{\kappa\la_-} \right.
		\end{equation}
		(the right-hand side quotient is defined thanks to Lemma \ref{GlobalWeyl}).
		
		By definition of $\overline{\Res}_J\mathbb U_\lambda^J$ and the first claim of Lemma \ref{GlobalWeyl} we have:
		
		\begin{equation} 
			\overline{\Res}_J\mathbb U_\lambda^J\simeq \mathbb W_\la\left/
			\sum_{\substack{s_\gamma \sigma_{\lambda}>\sigma_{\lambda}\\
					\gamma \in \Phi\backslash \Phi^J}} \right. \mathbb W_{s_\gamma \sigma_{\lambda}(\lambda_-)}.
		\end{equation}

		The weight $\lambda = \sigma_\la (\la_-)$ belongs to $ P^-_J$ (as well as $\la_-$ itself). Therefore $\sigma_\lambda$ is the minimal length element in the coset $W^J\sigma_\lambda$. We use the following claim:
		\begin{equation}\label{CosetInequality}
			\text{ if } \kappa > \sigma_\lambda, \text{ then } W^J\kappa > \sigma_\lambda \text{ for the whole coset}.
		\end{equation}
		In fact, we note that the extremal vector $u_{\kappa(\lambda)}$ belongs to $ \mathbb{W}_\lambda$ if and only if $\kappa \geq \sigma_{\lambda}$. Therefore \eqref{CosetInequality} follows from the second claim of Lemma \ref{GlobalWeyl}.
		
		Assume that $s_\gamma \sigma_\lambda >\sigma_\lambda$. Then there exists $\tau \in W^J$ such that $\tau s_\gamma \sigma_\lambda(\lambda_-)\in P^-_J$.  By \eqref{CosetInequality} we have that $\tau s_\gamma \sigma_\lambda>\sigma_\lambda$. Thus 
		$\mathbb{W}_{s_\gamma \sigma_\lambda(\lambda_-)}\subset 
		\mathbb{W}_{\tau s_\gamma \sigma_\lambda(\lambda_-)}$ and hence one gets a surjection of $\cI$-modules
		\begin{equation}\label{FirstSurjection}
			\overline{\Res}_J\mathbb U_\lambda^J\simeq \mathbb W_\la\left/
			\sum_{\substack{s_\gamma \sigma_{\lambda}>\sigma_{\lambda}\\
					\gamma \in \Phi\backslash \Phi^J}} \mathbb W_{s_\gamma \sigma_{\lambda}(\lambda_-)}\right. \to 
			\mathbb W_\la\left/
			\sum_{\substack{W\ni \kappa> \sigma_\la\\ \kappa\la_-\in P_J^-}} \mathbb W_{\kappa\la_-}\right.    
		\end{equation}
		(the right hand side here is the right hand side of \eqref{Uint}).
		
		Now let us show that for any $\kappa >\sigma_\lambda$, $W^J\kappa\geq \kappa$ there exists $\gamma \in \Phi \backslash \Phi_J$ such that 
		$\kappa \geq s_\gamma \sigma_\lambda >\sigma_\lambda$. This would imply the opposite surjection in \eqref{FirstSurjection}  
		and hence would prove the isomorphism \eqref{Uint}.
		
		Note that $\kappa =  s_{\gamma_{l(\kappa)-l(\sigma_\lambda)}}\dots s_{\gamma_1}\sigma_\lambda$ for some $\gamma_1, \dots, \gamma_{l(\kappa)-l(\sigma_\lambda)}$ such that
		\[l(s_{\gamma_j}\dots s_{\gamma_1}\sigma_\lambda)=l(\sigma_\lambda)+j.\]
		Note that not all $\gamma_i$ belong to $W^J$ and it is possible to choose a decomposition such that $\gamma_1 \in W \backslash W^J$.
		This implies $\kappa> s_{\gamma_1}\sigma_\lambda>\sigma_\lambda$
		and thus  
		$\mathbb{W}_{\kappa}\subset 
		\mathbb{W}_{s_\gamma \sigma_\lambda(\lambda_-)}$. 
		
		At this point we know that the map \eqref{FirstSurjection} is an isomorphism. The same arguments prove the isomorphism of $\cP_J$ modules:
		\begin{equation}\label{UIso}
			\mathbb U_\lambda^J\simeq \mathbb W_\la\left/
			\sum_{\substack{W\ni \kappa> \sigma_\la\\ \kappa\la_-\in P_J^-}} \mathbb W_{\kappa\la_-}\right.,    
		\end{equation}
		where the right hand side is seen as $\cP_J$ module thanks to Lemma \ref{GlobalWeyl}.
		Hence we obtain the desired isomorphism of $\cI$-modules via
		\[
		\overline{\Res}_J\mathbb U_\lambda^J\simeq 
		\mathbb W_\la\left/
		\sum_{\substack{W\ni \kappa> \sigma_\la\\ \kappa\la_-\in P_J^-}} \mathbb W_{\kappa\la_-}  \right.  \simeq  
		{\Res}_J \mathbb U_\lambda^J.
		\]
	\end{proof}

	\begin{prop}\label{URestriction}
		The restriction
		$\Res_{J}\mathbb U_\lambda^J$
		has an $\cI$-modules  filtration  $\mathcal{F}_{\tau(\lambda)}$, $\tau(\la)\in W^J.\la$ such that  
		\[\mathcal{F}_{\tau(\la)}\left/ \sum_{\eta(\la) > \tau(\la)} \mathcal{F}_{\eta(\la)} \right.\simeq \mathbb{U}_{\tau(\la)}^\varnothing.\]
	\end{prop}
	\begin{proof}
		We start with a realization of
		$\Res_{J}\mathbb U_\lambda^J$  
		as a generalized Weyl module with characteristics.
		
		Let ${\sigma'_\lambda} \in W$ be the maximal length element in the class $\sigma_\lambda\cdot {\rm stab}_W
		(\la_-)$. Then there exists a factorization $t _{\la_-} = ({\sigma'_\lambda}^{-1}w_0) u(\la)$ which
		refines to a reduced expression (see \cite{FKM} with the notation change $(\sigma'_\lambda)^{-1}w_0$ to  $m_\lambda$).
		More precisely, let us consider reduced expressions
		\begin{equation}\label{vu}
			(\sigma'_\lambda)^{-1}w_0=s_{i_1}\dots s_{i_r},\ u(\la)=\pi s_{i_{r+1}}\dots s_{i_M},
		\end{equation}
		and the reduced expression
		\begin{equation}
			{\mathbf t}_{\la_-}=\pi s_{\pi^{-1}i_1}\dots s_{\pi^{-1} i_r} s_{i_{r+1}}\dots s_{i_M}.
		\end{equation}
		In the same way as in the proof of Lemma \ref{DRestriction} (see formula \eqref{tbegining}) we get
		\begin{equation}\label{FirstBetas}
			\overline\beta^\vee_{j}({\mathbf t}_{\la_-})=s_{i_1}\dots s_{i_{j-1}}(-\alpha_{i_j}^\vee).
		\end{equation}
		Hence 
		\begin{equation}
			\{-\overline \beta_1^\vee, \dots, -\overline \beta_r^\vee\}=\{\beta \in \Phi_+, w_0{\sigma'_\lambda}(\beta)\in \Phi_-, \langle \beta, 
			\lambda\rangle\neq 0\}.
		\end{equation}
		
		Rewrite this equality as
		\[\{-\overline \beta_1^\vee, \dots, -\overline \beta_r^\vee\}=\{\beta \in \Phi_+, \sigma'_\lambda(\beta)\in \Phi_+, \langle \beta, 
		\lambda\rangle\neq 0\}.\]
		
		Our goal is to find reduced expression for $(\sigma'_\lambda)^{-1}w_0$ such that for $b=|\{\beta \in \Phi_+^J|\langle \beta,\lambda \rangle \neq 0\}|$ one has:
		\begin{equation}\label{needed}
			\{-\overline \beta_{r-b+1}^\vee,\dots,-\overline \beta_r^\vee\}=(\sigma'_\la)^{-1}\{\beta \in \Phi_+^J|\langle \beta,\lambda \rangle \neq 0\}.
		\end{equation}
		Thanks to $\sigma'_\la (\la_-) = \la\in P^J$, one has
		$(\sigma'_\la)^{-1}\{\beta \in \Phi_+^J|\langle \beta,\lambda \rangle \neq 0\}\subset \Phi_+$.
		
		Let $J^\vee=(-w_0).J$ (here $-w_0$ is understood as an automorphism of the Dynkin diagram.
		Note that there exists an element $w^J_\lambda\in W^{J^\vee}$ such that
		\begin{equation}\label{wlambdaJ}
			\Phi^{J^\vee+}\cap (w^J_\lambda)^{-1}(\Phi_{J^\vee-})=\{\beta \in w_0(\Phi_{J-})|\langle \beta,w_0(\lambda) \rangle \neq 0\}.
		\end{equation}
		In particular, $l(w_\la^J)=b$ and 
		\begin{equation}\label{w0wlaw0}
			w_0w_\lambda^Jw_0\{\beta \in \Phi_{J-}|\langle \beta,w_0(\lambda) \rangle \neq 0\}
			=\{\beta \in \Phi_{J+}|\langle \beta,\lambda \rangle \neq 0\}.
		\end{equation}
		
		By the construction $l(s_\beta w_0\sigma'_\lambda)<l(w_0\sigma'_\lambda)$ for any $\{\beta \in w_0\Phi_-^J|\langle \beta,w_0(\lambda) \rangle \neq 0\}$.
		Therefore $w_0\sigma'_\lambda=w_\lambda^J\chi$ for some $\chi \in W$ and this product refines a reduced decomposition.
		Let us  compute $\{-\overline \beta_{r-b+1}^\vee,\dots,-\overline \beta_r^\vee\}$ for such a reduces decomposition. By equation \eqref{FirstBetas} we have for $j= r-b+1, \dots, r$:
		
		\[-\overline\beta^\vee_{j}({\mathbf t}_{\la_-})=\chi^{-1}s_{i_{r-b+1}}\dots s_{i_{j-1}}(\alpha_{i_j}^\vee)=(\sigma'_\lambda)^{-1}w_0w_\lambda^J
		s_{i_{r-b+1}}\dots s_{i_{j-1}}(\alpha_{i_j}^\vee).\]
		Note that
		\[\{s_{i_{r-b+1}}\dots s_{i_{j-1}}(\alpha_{i_j}^\vee),\ j=r-b+1,\dots,r\}=\Phi_{J^\vee+}\cap (w^J_\lambda)^{-1}(\Phi_{J^\vee-}).\]
		which is the left hand side of \eqref{wlambdaJ}. Hence
		\begin{align*}
			\{-\overline\beta^\vee_{j}({\mathbf t}_{\la_-}),\ j=r-b+1,\dots,r\}=(\sigma'_\lambda)^{-1}w_0w_\lambda^J\{\beta \in w_0(\Phi_-^J)|\langle \beta,w_0(\lambda) \rangle \neq 0\}\\=
			(\sigma'_\lambda)^{-1}w_0w_\lambda^Jw_0\{\beta \in \Phi_-^J|\langle \beta,w_0(\lambda) \rangle \neq 0\}
			=(\sigma'_\la)^{-1}\{\beta \in \Phi_+^J|\langle \beta,\lambda \rangle \neq 0\}
		\end{align*}
		(see \eqref{w0wlaw0}). So we have constructed the decomposition \eqref{needed} with the desired properties.
		Therefore we have
		\begin{equation}\label{UWr}
			\Res_{J}\mathbb U_\lambda^J\simeq \mathbb W_\lambda(r-b).    
		\end{equation}
		
		Now let us apply the decomposition procedure to  
		$\mathbb W_\lambda(r-b)$ with respect to the roots $-\overline\beta^\vee_{j}({\mathbf t}_{\la_-}),\ j=r-b+1,\dots,r$ 
		(we are using the reduced decomposition of ${\mathbf t}_{\la_-}$ constructed above). Recall the equality
		$w_0\sigma'_\la = w_\la^J \chi$ with $l(w_\la^J)=b$, $l(\chi)=r-b$. We claim that
		\begin{itemize}
			\item the subquotients in the decomposition procedure are generated by the extremal vectors of weights $\tau(\la)\in W^J \la$,
			\item the subquotient corresponding to $\tau(\la)$ is isomorphic to $\mathbb U_{\tau\la}$.
		\end{itemize}
		
		By construction we have
		\[\sigma'_\la s_{\overline\beta_j}=s_{\gamma_j}\sigma'_\la,\ j=r-b+1,\dots,r\]
		for some $\gamma_j\in \Phi_J$ and
		\[s_{\gamma_{r-b+1}}\dots s_{\gamma_{r}}=w_0w_\la^Jw_0.\]
		
		Note that for $\tau \in W^J$, where $\tau$ is the minimal length element in the coset $\tau\cdot \rm{Stab}_{W^J}(\lambda)$ we have $l(\tau \sigma'_\lambda)=l(\tau)+l(\sigma'_\lambda)$.
		Therefore subquotients in our decomposition procedure correspond to sequences $\gamma_{j_1}, \dots, \gamma_{j_1}$, $r-b+1\le j_1<\dots <j_L\le r$ such that
		\[l(s_{\gamma_{j_1}}\dots s_{\gamma_{j_L}})=L.\]
		
		Note that there exists a one-to-one correspondence between such sequences and elements $\tau \in W^J$, $\tau\leq w_0w_\lambda^Jw_0$.
		Therefore these subquotients are generated by extremal vectors $u_{\kappa(\lambda_-)}=u_{\tau(\lambda)}$, $\tau\leq w_0w_\lambda^Jw_0$.
		
		Clearly the submodule $\mathcal{F}_{\tau (\la)}$ generated by $u_{\tau(\lambda)}$ contains the submodule $\mathcal{F}_{\tau' (\la)}$ generated  by $u_{\tau'(\lambda)}$ if and only if $\tau'\geq \tau$.
		By \cite{FMO1} the submodule  $\mathcal{F}_{\tau (\la)}$ of this decomposition procedure generated by $u_{\tau(\lambda)}$ is defined by the following relations:
		
		\begin{equation}
			\begin{cases}
				\widehat{\tau\sigma'_\la}(e_\alpha z^k)v_{\lambda}=0,~\text{if}~ \alpha \in \Phi_-, k=1, \dots\\
				\widehat{\tau\sigma'_\la}(e_\alpha)^{-\langle \alpha^\vee,\lambda_- \rangle+1}v_{\lambda}=0,~\text{if}~\alpha \in \Phi_+, \tau\sigma'_\lambda(\alpha) \in \Phi_-\cup \Phi_J\\
				\widehat{\tau\sigma'_\la}(e_\alpha)^{-\langle \alpha^\vee,\lambda_- \rangle}v_{\lambda}=0,~\text{if}~\alpha \in \Phi_+, \tau\sigma'_\lambda(\alpha) \in \Phi_+\cap (\Phi \backslash \Phi_J).
			\end{cases}      
		\end{equation}
		
		We are left to show that the quotient of 
		$\mathcal{F}_{\tau (\la)}$ 
		by the sum of $\mathcal{F}_{\tau' (\la)}$
		with $\tau'>\tau$ is isomorphic to $\mathbb{U}^\varnothing_{\tau(\la)}$.
		Recall the defining relations for $\mathbb{U}^\varnothing_{\tau(\la)}$:
		
		\begin{equation}
			\begin{cases}
				\widehat{\tau\sigma'_\la}(e_\alpha z^k)v_{\lambda}=0,~\text{if}~ \alpha \in \Phi_-, k=1, \dots\\
				\widehat{\tau\sigma'_\la}(e_\alpha)^{-\langle \alpha^\vee,\lambda_- \rangle+1}v_{\lambda}=0,~\text{if}~\alpha \in \Phi_+, \tau\sigma'_\lambda(\alpha) \in \Phi_-\\
				\widehat{\tau\sigma'_\la}(e_\alpha)^{-\langle \alpha^\vee,\lambda_- \rangle}v_{\lambda}=0,~\text{if}~\alpha \in \Phi_+, \tau\sigma'_\lambda(\alpha) \in \Phi_+.
			\end{cases}      
		\end{equation}
		Therefore  our claim is implied by the formula
		\[
		\U(\cI)\mathrm{span}\{\widehat{\tau\sigma'_\la}(e_\alpha)^{-\langle \alpha^\vee,\lambda_- \rangle+1}v_{\lambda},\  \alpha \in \Phi_+, \tau\sigma'_\lambda(\alpha) \in \Phi_J\cap \Phi_+\} = 
		\sum_{\tau'>\tau} 
		\mathcal{F}_{\tau' (\la)},
		\]
		which follows from the construction above.
	\end{proof}

	\begin{cor}\label{CostandardRestriction}
		The restriction
		$\Res_{J}\mathbb U_\lambda^J$
		has an $\cI$-modules  filtration  $\mathcal{F}_{\tau(\lambda_-)}$, $\tau\in W^J \sigma_\lambda$ such that  
		\[\mathcal{F}_{\tau(\lambda_-)}\left/ \sum_{\eta > \tau} \mathcal{F}_{\eta(\lambda_-)} \right.\simeq (\nabla^{\varnothing}_{-\tau(\lambda_-)})^\vee.\]
	\end{cor}
	
	\begin{cor}
		If $\Phi^J=\Phi$, then $\fp^J=\fg[t]$, $\la=\la_-\in \Phi_-$ and $\mathbb U^J_\la$ coincides with the global Weyl module $\mathbb W_\la$. 
		Hence $\mathbb W_\la$ admits a filtration by $\cI$-modules such that the associated graded space is a direct sum of modules of the form $\mathbb U_{\tau(\la)}$.
	\end{cor}
	
	We close this subsection with two lemmas, which are needed for the proof of the Peter-Weyl-van der Kallen theorem for parahoric groups.
	
	\begin{lem}\label{EndRestriction}
		One has the isomorphisms between the endomorphism algebras 
		\begin{gather*}
			{\rm End}_{\cP_J} (\mathbb D^J_\la)\simeq {\rm End}_{\cI}(\Res_J \mathbb D^J_\la)\simeq {\rm End}_\cI (\mathbb D^\varnothing_\la),\\
			{\rm End}_{\cP_J}(\mathbb U^J_{-w_0^J\la})\simeq {\rm End}_\cI(\Res_J \mathbb U^J_{-w_0^J\la}).
		\end{gather*}
	\end{lem}
	\begin{proof}
		The $\cP_J$ modules $\mathbb D_\la, \mathbb U^J_{-w_0^J\la}$ and their restrictions are generated by the extremal weight components and are cyclic. 
		Therefore the endomorphism algebras are equal to the quotients of $\U(\fh[z])$ acting effectively on the generators. Thus we obtain the desired isomorphisms.
	\end{proof}
	
	\begin{lem}\label{HWA}
		The endomorphism algebras of $\mathbb D^J_\la$ and $\mathbb U^J_{-w_0^J\la}$ are isomorphic to the polynomial algebra 
		$\mathcal{A}_\lambda^D={\rm End}_\cI (\mathbb D^\varnothing_\la)$ from ~\cite[Proposition 4.15]{FKhMO}.
	\end{lem}
	\begin{proof}
		The claim for $\mathbb D^J_\la$ is implied by Lemma \ref{EndRestriction}, since by definition  $\mathcal{A}_\lambda^D = {\rm End}_\cI (\mathbb D^\varnothing_\la)$. Let us work out the case of $\mathbb U^J_{-w_0^J\la}$.
		
		Using Proposition \ref{URestriction} we get
		$\Res_{\cI}\mathbb U_{-w_0^J\la}^J\supset \mathbb U_{-\la}^\varnothing$ and the quotient module 
		$\Res_{\cI}\mathbb U_{-w_0^J\la}^J\left/\mathbb U_{-\la}^\varnothing\right.$ has trivial $\lambda$ weight component. Therefore the algebra of endomorphisms of 
		$\Res_{\cI}\mathbb U_{-w_0^J\la}^J$ acts on the $-\lambda$ weight component of $\mathbb U_{-\la}^\varnothing$ as on the free module of rank one. Hence this algebra is isomorphic to the algebra of endomorphisms of $\mathbb U_{-\la}^\varnothing$, which is isomorphic to the algebra of endomorphisms of $\mathbb D^\varnothing_\la$.
		Now Lemma \ref{EndRestriction} completes the proof.
			\end{proof}

	\subsection{Comparison with standard and costandard modules}
	
	\begin{lem}
		\label{Categories}
		One has
		\begin{itemize}  
			\item $\mathbb D_\lambda^J, D_\lambda^J\in {\fC^{J}}^{\vee}_{\preceq \lambda}$;
			\item $(\mathbb D_\lambda^J)^\vee, (D_\lambda^J)^\vee\in {\fC^{J}}^{\vee}_{\preceq^\vee \lambda}$;
			\item $\mathbb U_\lambda^J, U_\lambda^J\in {\fC^{J}}^{\vee}_{\preceq^\vee \lambda}$; 
			\item $(\mathbb U_\lambda^J)^{\vee}, (U_\lambda^J)^{\vee}\in {\fC^{J}}
			_{\preceq \lambda}$.
		\end{itemize}
	\end{lem}
	\begin{proof}
		To prove the first and the third claims we have to show that all $J$-antidominant weights of the modules $\mathbb D_\lambda^J, D_\lambda^J$ are $\preceq$ than $\lambda$ and all $J$-antidominant weights of the modules   $\mathbb U_\lambda^J, U_\lambda^J$ are $\preceq^\vee$ than $\lambda$. 
		All the weights of $\Delta_\lambda^\varnothing$ are $\preceq$ than $\lambda$ and the first claim follows from Corollary \ref{StandardRestriction}.
		In the same way all the weights of $(\nabla^{\varnothing}_{-\tau(\lambda_-)})^\vee$ are $\preceq^\vee$ than $\lambda$. Therefore Corollary \ref{CostandardRestriction} implies the third claim.
		The second and the fourth claims are equivalent to the first and the third claims by definition of the orders. 
	\end{proof}

	\begin{thm}\label{StandardCostandardCh}
		The standard and costandard modules in the categories 
		$\mathcal{C}_{\preceq \lambda}^{J}$ and $\mathcal{C}_{\preceq \lambda}^{J\vee}$ admit the following explicit description as $\mathcal{P}_J$ modules:
		\[
		\Delta_\lambda^J \simeq \mathbb D_\lambda^J,\quad \overline{\Delta}_\lambda^J \simeq  D_\lambda^J,\quad   \nabla_{\lambda}^J \simeq (\mathbb{U}_{-w_0^J\lambda}^J)^\vee,\quad \overline{\nabla}_\lambda^J \simeq (U_{-w_0^J\lambda}^J)^\vee.
		\]
	\end{thm}
	\begin{proof}
		Assume that $M$ is a cyclic $\mathcal{P}^J$ module generated from an element $v_\la$ of weight $\lambda$ such that all weights of $M$ are 
		$\preceq$ to $\lambda$. Then $M$ admits the surjective map $\mathbb D_\lambda^J\rightarrow M$, $v_\lambda \mapsto m_\lambda$. 
		Therefore $\mathbb D_\lambda^J$ and $D_\lambda^J$ are projective in the subcategories $\mathcal{C}_{\preceq \lambda}^{J\vee}$ 
		and $\overline{\mathcal{C}}_{\preceq \lambda}^{J\vee}$ respectively
		(recall that $\overline{\mathcal{C}}_{\preceq \lambda}^{J\vee}$ consists of modules $N$ such that $[N:L_\la]\le 1$).
		Hence $\mathbb D_\lambda^J \simeq \Delta_\lambda^J,\ D_\lambda^J \simeq \overline{\Delta}_\lambda^J$.
		
		Similarly by the same arguments for the order $\preceq^\vee$ we get that $\mathbb U_\lambda^J$ and $U_\lambda^J$ are projective in the subcategories $\mathcal{C}_{\preceq^\vee \lambda}^{J\vee}$ 
		and $\overline{\mathcal{C}}_{\preceq^\vee \lambda}^{J\vee}$ respectively.
		Hence by dualization, we get the second part of the theorem (we note that $-w_0^J\la\in P_J^-$ as soon as $\la\in P_J^-$).
	\end{proof}

	Now let us formulate a similar theorem (with a similar proof) describing the standard and costandard modules in the categories 
	$\mathcal{C}_{\preceq^\vee \lambda}^{J}$ and 
	$\mathcal{C}_{\preceq^\vee \lambda}^{J\vee}$. 
	
	\begin{thm}\label{StandardCostandardChDual}
		The standard and costandard modules in the categories 
		$\mathcal{C}_{\preceq^\vee \lambda}^{J}$ and 
		$\mathcal{C}_{\preceq^\vee \lambda}^{J\vee}$ admit the following explicit description as $\mathcal{P}_J$ modules:
		\[
		\Delta_\lambda^J \simeq \mathbb U_\lambda^J,\quad \overline{\Delta}_\lambda^J \simeq  U_\lambda^J,\quad   \nabla_{\lambda}^J \simeq (\mathbb{D}_{-w_0^J\lambda}^J)^\vee,\quad \overline{\nabla}_\lambda^J \simeq (D_{-w_0^J\lambda}^J)^\vee.
		\]
	\end{thm}
	\begin{proof}
		We note that by Lemma \ref{Categories} $\mathbb U_\lambda^J\in \mathcal{C}_{\preceq^\vee \lambda}^{J}$  and $(\mathbb{D}_{-w_0^J\lambda}^J)^\vee\in \mathcal{C}_{\preceq^\vee \lambda}^{J\vee}$. Now the proof proceeds along the lines of the proof of Theorem \ref{StandardCostandardCh}.
	\end{proof}

	\begin{thm}
		\begin{itemize}
			\item $\Res_{J}\Delta_\lambda^J\simeq \Delta_\lambda^\varnothing$, $\Res_{J}\overline{\Delta}_\lambda^J\simeq \overline{\Delta}_\lambda^\varnothing$.
			\item $\Res_{J}\nabla_\lambda^J$ has the filtration $\mathcal{F}_{\tau(\lambda)}$, $\tau \in W^J$, such that the associated graded module is
			isomorphic to $\bigoplus \nabla_{\tau (\lambda)}^\varnothing$.
		\end{itemize}
		
	\end{thm}
	\begin{proof}
		Immediate corollary from  Corollary \ref{StandardRestriction} and Proposition \ref{URestriction}. 
	\end{proof}
	
	\begin{cor}\label{Ext::Vanishing::Iwahori}
		For $\lambda \neq \nu \in P_J^-$ one has 
		\[
		\Ext_{\cI}^{\udot}(\Res_{J}\overline{\Delta}_\lambda^J,\Res_{J}\nabla_\nu^J)=0.
		\]
	\end{cor}
	\begin{proof}
		It was shown in \cite{FKhMO} that 
		\[\Ext_{\cI}^{\udot}(\overline{\Delta}_\lambda^\varnothing,\nabla_\eta^\varnothing)=0,\qquad \lambda,\eta \in P.
		\]
		Now the claim follows from the long exact sequence of $\Ext$'s.
	\end{proof}

	\section{Main results}
	\label{Main}
	In this section, we prove the main theorems of our paper.
	\subsection{Parahoric categories}
	Recall that the irreducible modules $L_{\la,k}$ in the categories $\fC^{J}$ and ${\fC^{J}}^{\vee}$  are indexed by the $J$-antidominant weight $\lambda\in P_J^{-}$ and an integer $k$ ($L_{\lambda,k}$ consists of the modules $L_\la$ placed in degree $k$).
	Consider the projection on the first factor 
	$$\rho: Irreps(\fC^J) = P_J^{-}\times \bZ \twoheadrightarrow P_J^{-}. $$
	Following Definition~\ref{def::Cherednik::ordering} we have two standard partial orders on $P_J^{-}$ the Cherednik order $\preceq$ and the dual Cherednik order $\prec^{\vee}$. The goal of this section is to show that $(P_J^{-},\prec)$-ordered and $(P_J^{-},\prec^{\vee})$-ordered categories $\fC^{J}$ and ${\fC^{J}}^{\vee}$ are stratified in the sense of Definition~\ref{dfn::stratified}.
	The four cases (two categories and two orders) are divided into two conceptually different pairs. We start with $\fC^J$ and $\prec$.
	
	\begin{thm}
		\label{thm::fCJ::stratified}
		For all $J$ the $(P_J^{-},\prec)$-ordered category $\fC^{J}$ of modules over parahoric Lie algebra $\cP_{J}$ is stratified.
	\end{thm}
	\begin{proof}
		To show that the category $\fC^{J}$ is stratified for the Cherednik partial order (Definition~\ref{def::Cherednik::ordering}) it is enough to show the vanishing of the derived homomorphisms between the corresponding proper standard modules $\overline{\Delta_{\lambda,k}^{J}}$ and costandard modules $\nabla^{J}_{\mu,k}$.
		Note that the index $k$ does not play a role, since the (proper) (co)standard modules are isomorphic up to a grading shift:
		$$
		\forall \lambda\in P_J^{-} \ \ 
		\overline{\Delta}_{\lambda,0}^J\{k\} = \overline{\Delta}_{\lambda,k}^J, \ \ 
		\nabla_{\lambda,0}^{J}\{k\} = \nabla_{\lambda,k}^{J} \ \ 
		$$
		(the $n$-th graded component of the module $M\{k\}$ is equal to the $(n-k)$-th graded component of  $M$). We omit the second index $k$ if no confusion is possible.
		
		Recall  the  description of the proper standard modules $\overline{\Delta}^J_{\lambda}$ and costandard modules $\nabla^J_{\mu}$ in terms of 
		generators and relations and  their description as $\cI$-modules from Section~\ref{sec::standards}, (see Corollary \ref{StandardRestriction} and Corollary \ref{CostandardRestriction}).
		Below we use the main consequence (Corollary~\ref{Ext::Vanishing::Iwahori}) for the proof of Theorem~\ref{thm::fCJ::stratified}.
		
		For all pairs of different $P_J$-antidominant weights $\lambda$ and $\mu$ we have
		\begin{equation}
			\label{eq::Ext::vanishing}
			\Ext_{\fC^{J}}(\overline{\Delta}^{J}_{\lambda},\nabla^{J}_{\mu}) = 
			\Ext_{\fC^{J}}(\LInd_{J}\circ \Res_J\bar\Delta^{J}_{\lambda}, \nabla^{J}_{\mu}) =
			\Ext_{\fC}(\Res_J\overline{\Delta}^{J}_{\lambda},
			\Res_J\nabla^{J}_{\mu} ) = 0.
		\end{equation}
		The first equality follows from the fact that $\overline{\Delta}^{J}_{\lambda}$ is $\cP_J$-integrable module and the composition of induction and restriction functors on $\cP_J$-integrable modules coincides with the identity functor (see~\cite{KKhM}). 
		The second equality follows from the adjunction of restriction and induction functors. Finally, the vanishing of the derived homomorphisms in the category $\fC^J$ follows from Corollary~\ref{Ext::Vanishing::Iwahori}.
	\end{proof}

	Recall that thanks to Lemma~\ref{HWA} and Theorem \ref{StandardCostandardCh} 
	we know that the endomorphism algebras of  the standard module $\Delta_\la^{J}$ and of 
	(the linear dual to) the costandard module $\nabla_{\lambda}$ 
	are isomorphic to a polynomial algebra  by~$\mathcal{A}_{\la}^{D}$.
	Following the notation of \cite{FKhMO} we denote the character of $\mathcal{A}_{\la}^{D}$ by $a_\lambda(q)$.
	\begin{cor}
		\label{(co)stch}
		The characters of (proper) (co)standard modules in $\fC^J$ with respect to the Cherednik order are given by
		\[
		[\overline{\Delta}_\la^{J}] = E_{\lambda}^{J}(X;q,0), \qquad 
		[\overline{\nabla}_\la^{J}] = 
		E_{\lambda}^{J}(X;q,\infty).
		\]
		\[
		[{\Delta}_\la^{J}] ={a_{\lambda}(q)}{E_{\lambda}^{J}(X;q,0)}, \qquad 
		[{\nabla}_\la^{J}] = 
		{a_\lambda(q^{-1})}{E_{\lambda}^{J}(X;q,\infty)}.\]
	\end{cor}
	\begin{proof}
		Since category $\fC^{J}$ is stratified the characters of proper standard and costandard modules form an orthogonal basis for the Ext-pairing. Thanks to Proposition~\ref{prop::pairing::Ext} we know that this pairing coincides with the pairing $\langle  \cdot,\cdot\rangle^J_{t=0}$ ~\eqref{pairing::J}.
		Moreover, we know that the character of the proper costandard module differs from the character of the costandard module by a rational function in $q$,
		because their restrictions to $\cI$ are generalized Weyl modules with characteristic {(see \cite{FMO1})}.
		Hence the characters of proper standard and proper costandard modules are orthogonal with respect to the pairing~\eqref{pairing::J}:
		\begin{equation}
			\forall \lambda\neq \mu\in P_{J}^{-} \quad
			\langle[\overline{\Delta}_{\lambda}^{J}],[\overline{\nabla}_{\mu}^{J}]\rangle_{t=0} = 0.
		\end{equation}
		Moreover, from the definition of the proper standard module $\overline{\Delta}_{\lambda}^{J}$ we know that it belongs to $\fC^{J}_{\preceq\lambda}$ and the multiplicity of $L_{\lambda}$ in $\overline{\Delta}_{\lambda}^{J}$ is equal to $1$. We conclude that the characters of the proper standard module belong to 
		$\bZ(q)[P]^{W^J}$ and satisfy the following properties:
		$$[\overline{\Delta}_{\lambda}^{J}] = m_{\lambda}^{J} + \sum_{\mu\prec \lambda} c_{\lambda}^{\mu} m_{\mu}^{J}$$
		($m_{\lambda}^{J}$ is the monomial basis in $\bZ(q)[P]^{W^J}$).
		Similarly, the character of the proper costandard module
		also belongs to the same ring $\bZ(q)[P]^{W^J}$ and satisfies the same upper-triangular decomposition:
		$$[\overline{\nabla}_{\lambda}^{J}] = m_{\lambda}^{J} + \sum_{\mu\prec \lambda} d_{\lambda}^{\mu} m_{\mu}^{J}.$$
		The orthogonality implies the orthogonality with all monomials $m_{\mu}^{J}$ with $\mu\prec \lambda$
		$$
		\langle [\overline{\Delta}_{\lambda}^{J}], m_{\mu}^{J}\rangle_{t=0} = 
		\langle m_{\mu}^{J}, [\overline{\nabla}_{\lambda}^{J}] \rangle_{t=0}=0.
		$$
		Thus, the characters of proper (co)standard modules coincide with the result of the Gram-Shmidt orthogonalization process of the monomial basis for the Cherednik partial order. 
		Therefore they coincide with the parasymmetric Macdonald polynomials $E_{\lambda}^{J}$ after appropriate substitution $t=0$ or $t=\infty$ respectively (see Definition~\ref{Parasymmetric}).
	\end{proof}
	
	\begin{cor}
		\label{cor::DU::characters}
		One has the character formulas:
		\begin{gather*}
			[D_\lambda]= E_{\lambda}^{J}(X,q,0), \quad [U_{-w_0^J\lambda}]=E_{\lambda}^{J}(X^{-1},q^{-1},\infty),\\
			[\mathbb D_\lambda]= a_\la(q)E_{\lambda}^{J}(X,q,0), \quad [\mathbb U_{-w_0^J\lambda}]=a_\la(q) E_{\lambda}^{J}(X^{-1},q^{-1},\infty).
		\end{gather*}  
	\end{cor}

	\begin{cor}
		The polynomials $E_\la^J(X,q,0)$ and $E_\la^J(X,q^{-1},\infty)$    can be written as sums of characters of Levi subalgebra $\fl_J\subset \fp_J$  with coefficients in $\bZ_{\ge 0}[[q]]$ (the generalized Schur positivity).
	\end{cor}
	
	\begin{cor}
		The parasymmetric Macdonald polynomials satisfy the relations
		\[
		E_{\lambda}^{J}(X,q,0) = E_{\lambda}(X,q,0),\qquad 
		E_{\lambda}^{J}(X,q^{-1},\infty) = \sum_{w\in W^J} \frac{a_{w\lambda}(q)}{a_\la(q)}E_{w\lambda}(X,q^{-1},\infty).
		\]   
	\end{cor}
	\begin{proof}
		It follows from Corollary \ref{cor::DU::characters}, Corollary \ref{StandardRestriction} and  Proposition \ref{URestriction}.
	\end{proof}

	\begin{thm}\label{catarestrat}
		The $P_J^{-}$-ordered categories $\fC^{J}$ and ${\fC^{J}}^{\vee}$ are stratified for each of the two orders on $P_J^{-}$: the Cherednik order $\prec$ and the dual Cherednik order $\prec^{\vee}$.
	\end{thm}
	\begin{proof}
		We already proved that $\fC^{J}$ is stratified for the Cherednik order in Theorem~\ref{thm::fCJ::stratified}. 
		Taking the graded linear dual we see that the dual category ${\fC^{J}}^{\vee}$ is stratified for the dual Cherednik order.
		
		Let us show that the category $\fC^{J}$ for the dual Cherednik order $\prec^{\vee}$ is also stratified.
		For that we use the description of the proper standard and costandard modules from Theorem~\ref{StandardCostandardChDual}. 
		Theorems ~\ref{StandardCostandardCh} and  ~\ref{StandardCostandardChDual} and Corollary \ref{(co)stch} imply that   
		$$
		[\overline{\Delta}_{\lambda}^{J}] = E_{-w_0^{J}\lambda}^{J}(X^{-1};q^{-1},\infty),
		\quad [\overline{\nabla}_{\lambda}^{J}] = E_{-w_0^J\lambda}^J(X^{-1};q^{-1},0).
		$$
		Hence, for $\lambda\neq\mu$ we have
		$$
		\langle[\overline{\Delta}_{\lambda}^{J}], [\overline{\nabla}_{\mu}^{J}]\rangle_{t=0}^J = 
		\langle[\overline{\Delta}_{\lambda}^{J}], [{\nabla_{\mu}^{J}}]\rangle_{t=0}^J = 0.
		$$
		Finally, thanks to Theorem~1.22 from ~\cite{FKhMO} we know that 
		vanishing of the Euler characteristic of Ext's (see pairing~\eqref{eq::Ext::pairing::1} and Proposition~\ref{prop::pairing::Ext}) implies vanishing of all higher derived homomorphism between proper standards and costandards. The latter is sufficient for the category to be stratified.
		
		Finally, since $\fC^{J}$ is stratified with respect to $\prec^{\vee}$
		the remaining pair ${\fC^{J}}^{\vee}$, $\prec$ is also stratified (taking the linearly graded dual modules).
	\end{proof}
	
	It was shown in \cite{FKhMO}, Theorem 2.26 that if the categories $\fC(\fa)$ and $\fC^\vee(\fa)$ are stratified (for a Lie algebra $\fa$ as in Section \ref{sec::parabolic::setup}),
	then under certain assumptions, an analog of the  Peter-Weyl theorem holds, describing the bi-module of functions on the Lie group of $\fa$. 
	Below we prove such a theorem for parahoric groups $\mathbf{P}_J$.
	
	\begin{thm}\label{PWvdKparahoric}
		There exists a canonical increasing filtration $\mathcal{F}_\lambda$, $\lambda \in P_J^-$ of the bimodule of algebraic functions on the simple-connected Parahoric group  such that 
		\[\mathcal{F}_\lambda\left/ \sum_{\mu \prec \lambda}\mathcal{F}_\mu \right.=\nabla_\lambda^J \otimes_{\mathcal{A}^D_\lambda}((\Delta_\lambda^J)^{\vee})^o.\]
	\end{thm}
	\begin{proof}
		Thanks to \cite{FKhMO}, Theorem 2.26 it suffices to prove that the endomorphism algebras 
		of $\Delta_\la^J$ and of $(\nabla^J_\la)^\vee$ are isomorphic.
		These endomorphism algebras are the endomorphism algebras of the modules $\mathbb D^J_\la$ and $\mathbb U^J_{-w_0^J\la}$. The desired isomorphism is provided by Lemma \ref{HWA}. 
	\end{proof}

	\subsection{Parabolic categories}
	The main part of the paper deals with (infinite-dimensional)  parahoric subalgebras of the affine Kac-Moody Lie algebras.
	However, most of the results hold true for the finite-dimensional parabolic Lie algebras as well. To the best of our knowledge, these results are new in the finite-dimensional situation as well, so we discuss them in detail in this section.
	
	Let us adopt the setup from Section \ref{sec::parabolic::setup}.
	As before, $\fg$ is a simple finite-dimensional Lie algebra of rank $n$ with a fixed Borel subalgebra $\fb$. We fix a subset $J$ of the indexing set of simple roots. 
	Let $\fp_{J}\subset\fg$ be the corresponding parabolic subalgebra with the Levi subalgebra $\fl_{J}$.
	Let $\fC_{0}^J:=\fC(\fp_{J})$ and ${\fC_0^{J}}^{\vee}=\fC^{\vee}(\fp_{J})$ be the corresponding categories of $\fp_J$-modules defined in~\S\ref{GeneralFramework} for general Lie algebra $\fa$.
	Note that $\fp_J$ has trivial grading (all elements are of degree 0), so the categories $\fC_0^{J}$ (respectively ${\fC_0^{J}}^{\vee}$) consist of finitely cogenerated (resp. generated) $\fp_J$-modules $M$ such that $M$ is $\fl_J$-integrable, $\fh$-graded with finite-dimensional weight subspaces.
	The irreducible modules are indexed by $\fl_J$-antidominant weights $P_J^-$ and we can restrict Cherednik partial order $\prec$ and its dual $\prec^{\vee}$ to this set.
	The projective cover $\bP_{\lambda}$ of $L_\la$ in ${\fC_0^{J}}^{\vee}$ coincides with the restriction of the parabolic Verma module $M_{\lambda}:=\Ind_{\fp_{J}^{-}}^{\fg} L_{\lambda}$ and {the injective hull $\bI_{\lambda}$ is  the dual module}. In particular, the multiplicity of $L_\lambda$ in its projective cover $\bP_{\lambda}$ and in its injective hull $\bI_{\lambda}$ is equal to one.
	
	Thanks to the BGG-resolutions we conclude that $\fC_0^{J}$ has enough injectives (resp. ${\fC_0^{J}}^{\vee}$ has enough projectives) and each irreducible admits a resolution of length no more than $l(w_0)$. Hence each module has a finite homological dimension and we consider the bounded derived categories $\bD^b(\fC_0^{J})$ and $\bD^b({\fC_0^{J}}^{\vee})$.
	
	\begin{cor}
		The $Ext$-pairing in $\bD^b(\fC_0^{J})$ and in $\bD^b({\fC_0^{J}}^{\vee})$
		categorifies the constant term pairing on the ring of $W^J$-symmetric functions $\bZ[P]^{W^J}$. That is, for finite-dimensional $M,N\in\fC^{J}_0$ we have
		\begin{equation}
			\sum_{i}(-1)^{i} \dim\Ext^i_{\fC^{J}_0}(M,N) = \langle f, g \rangle^{J}_{t=q=0}, 
			\label{eq::pairing::finite}
		\end{equation}
		where   $\langle f, g \rangle^{J}_{q=t=0}= \left[f(x)g(x^{-1}) \prod_{\alpha\in\Phi^J_{-}\cup\Phi_{+}}(1-X^{\alpha}) \right]_{1}.$
	\end{cor}
	\begin{proof}
		This is the  $q=0$ specialization of  Proposition~\ref{prop::pairing::Ext}.
	\end{proof}
	
	Note that the specializations  $\{E_{\lambda}(X,q,0)_{q=0} |\ \lambda\in P_{J}^{-}\}$ and 
	$\{E_{\nu}(X,q,\infty)_{q=\infty} |\ \nu\in P_{J}^{-}\}$ form a dual orthogonal basis with respect to the pairing~\eqref{eq::pairing::finite}.
	\begin{thm}\label{pararehwc}
		The category $\fC_0^{J}$ is the highest weight category for the Cherednik partial order $\prec$ on the set of antidominant weights. Standard and costandard modules are finite-dimensional and their characters are given by 
		\begin{gather*}
			[\Delta_{\lambda}^{J,fin}]= [\overline{\Delta}_{\lambda}^{J,fin}] = E_{\lambda}^{J}(X;q,0)_{q=0}, \\
			[\nabla_{\lambda}^{J,fin}]= [\overline{\nabla}_{\lambda}^{J,fin}] = E_{\lambda}^{J}(X;q,\infty)_{q=\infty}.    
		\end{gather*}
		
		The category ${\fC_0^{J}}$ is the highest weight category for the dual Cherednik order $\prec^{\vee}$ on the indexing set $P_J^{-}$ of the set of simple objects in $\fC_0^{J}$. 
		The characters of the standard and costandard modules are given by 
		\begin{gather*}
			[\Delta_{\lambda}^{J,fin}]= [\overline{\Delta}_{\lambda}^{J,fin}] = E_{-w_0^J\lambda}^{J}(X^{-1};q,\infty)_{q=\infty}, \\
			[\nabla_{\lambda}^{J,fin}]= [\overline{\nabla}_{\lambda}^{J,fin}] =E_{-w_0^J\lambda}^{J}(X^{-1};q,0)_{q=0}.    
		\end{gather*}
			The category ${\fC_0^{J}}^{\vee}$ is the highest weight category with respect to the Cherednik and dual Cherednik orders. The standard and costandard objects are finite-dimensional and coincide with the corresponding ones in $\fC^{J}$.
	\end{thm}
	\begin{proof}
		Note that the Lie algebra $\cP^{J}$ is positively graded and therefore the standard module $\Delta_{\lambda,fin}^J$ in ${\fC^{J}_{0}}^{\vee}$ is the degree zero component of the $\cP^{J}$-module $\Delta_{\lambda}^J\in\fC^{J}$.
		The same happens with the (proper) costandard $\cP^{J}$ and $\fp^{J}$-modules.
		Consequently, we know that the characters of standard and costandard modules are finite and form a pair of dual bases. Now Theorem 1.22, ~\cite{FKhMO} implies that the categories are stratified.
	\end{proof}
	
	Let us also derive the parabolic analog of the Peter-Weyl theorem for simply connected parabolic group $\mathbf{P}_J\subset \mathbf{G}$ whose Lie algebra is $\fp_J$. For the  Borel subgroup 
	($J=\emptyset$)  the theorem below is due to van der Kallen~\cite{vdK}.
	
	\begin{cor}\label{PWvdKparabolic}
		For each $J\subset\{1,\dots,n\}$ there exists a canonical increasing filtration $\mathcal{F}_\lambda$, $\lambda \in P_J^-$ of the bimodule of algebraic functions on the simple-connected parabolic group   such that 
		\[\mathcal{F}_\lambda\left/ \sum_{\mu \prec \lambda}\mathcal{F}_\mu \right.=\nabla^{J,fin}_\lambda \otimes ((\Delta^{J,fin}_\lambda)^{\vee})^o.\]
	\end{cor}
	
	\begin{rem}
		One can write down explicitly the defining relations for the cyclic $\fp_J$ modules   $\Delta_\la^{J,fin}$ and $\nabla_\la^{J,fin}$. Namely, 
		each relation in  Definitions \ref{Ddefrel} and \ref{Udefrel} is a power of a root vector. Then the defining relations
		in the parabolic case are obtained from the relations in Definitions \ref{Ddefrel} and \ref{Udefrel} by keeping only the powers of elements from the Lie algebra $\fg$.   
		
		We note also that the restriction of 
		$\Delta_\la^{J,fin}$ to the Iwahori subalgebra is isomorphic to a Demazure module, and $\nabla_\la^{J,fin}$ is filtered by the Demazure atoms.
	\end{rem}


\begin{thebibliography}{XXXXXX}
		
		\bibitem[AE]{AE}
		O. Azenhas and A. Emami,
		{\it An analogue of the Robinson-Schensted-Knuth correspondence and non-symmetric Cauchy kernels for truncated staircases},
		European J. Combin. 46 (2015), 16--44. 
		
		\bibitem[AGL]{AGL}
		O. Azenhas, T. Gobet, and C. Lecouvey,
		\emph{Non symmetric Cauchy kernel, crystals and last passage percolation}, arXiv:2212.06587.
		
		\bibitem[Bar1]{Bar1}
		W. Baratta, \emph{Some properties of Macdonald polynomials with prescribed symmetry}, Kyushu J. Math. 64 (2010), 323--343.
		
		\bibitem[Bar2]{Bar2}
		W. Baratta, \emph{Computing nonsymmetric and interpolation Macdonald polynomials}, 
		arXiv:1201.4450.
		
		\bibitem[BrFi]{BrFi}
		A. Braverman, M. Finkelberg, {\it Weyl modules and $q$-Whittaker functions}, Mathematische Annalen, vol. 359 (1),
		2014, pp 45--59.
		
		
		\bibitem[BB]{BB}
		A. Bj\"{o}rner and F. Brenti, Combinatorics of Coxeter groups,
		Graduate Texts in Mathematics Vol.~231, 
		Springer, New York, 2005.
		
		
		
		
		\bibitem[Br]{Br}
		J.~Brundan, {\it Graded triangular bases}, arXiv:2305.05122. 
		
		\bibitem[BS]{BS}
		J.~Brundan, C.~Stroppel, {\it Semi-infinite highest weight categories},  	arXiv:1808.08022.
		
		\bibitem[BWW]{BWW}
		J.~Brundan, W.~Wang, B.~Webster, {\it Nil-Brauer categorifies the split $\imath$-quantum group of rank one}, 
		arXiv:2305.05877.
		
		\bibitem[BBCKhL]{BBCKhL}
		M.~Bennett, A.~Berenstein, V.~Chari, A.~Khoroshkin, S.~Loktev, {\it Macdonald polynomials and BGG reciprocity for current algebras},  Selecta Math. (N.S.) 20 (2014), no. 2, 585--607. 
		
		\bibitem[CI]{CI}
		V. Chari, B. Ion, {\it BGG reciprocity for current algebras}, Compositio Mathematica 151 (2015), pp. 1265--1287.
		
		\bibitem [Ch1] {Ch1}
		{I.~Cherednik},
		{\it Nonsymmetric Macdonald polynomials},
		IMRN {10} (1995), 483--515.
		
		\bibitem [Ch2] {Ch2}
		I.~Cherednik,
		{\it Double affine Hecke algebras},
		London Mathematical Society Lecture
		Note Series, {319}, Cambridge University Press, Cambridge, 2006.
		
		
		
		\bibitem[CFK]{CFK}
		V.~Chari, G.~Fourier, and T.~Khandai,
		{\it A categorical approach to Weyl modules},
		Transform. Groups, 15(3):517--549, 2010.
		
		\bibitem[CG]{CG}
		V. Chari and J. Greenstein,
		{\it Current algebras, highest weight categories and quivers},
		Adv. Math. 216 (2007), 811--840.
		
		\bibitem [CO1] {CO1}
		I.~Cherednik and {D.~Orr},
		{\it Nonsymmetric difference Whittaker functions}, 
		Math. Z. 279 (2015), no. 3--4, 879--938.
		
		\bibitem [CO2] {CO2}
		I.~Cherednik and {D.~Orr},
		{\it One-dimensional nil-DAHA and Whittaker functions},
		Transformation Groups { 18}:1 (2013), 23--59.
		
		\bibitem[CK]{CK}
		S.-I. Choi and J.-H. Kwon,
		{\it Lakshmibai-Seshadri paths and non-symmetric Cauchy identity},
		Algebr. Represent. Theory 21 (2018), no. 6, 1381-1394. 
		
		\bibitem[CPS]{CPS}
		E. Cline, B. Parshall and L. Scott, {\it Finite dimensional algebras and highest weight categories},
		J. Reine Angew. Math. 391 (1988), 85–99.
		
		\bibitem[Dlab]{Dlab}
		V. Dlab, {\it Properly stratified algebras}, C. R. Acad. Sci. Paris 331 (2000), 191--196.
		
		
		
		
		\bibitem[FeMa]{FeMa}
		E. Feigin, I. Makedonskyi, {\it Generalized Weyl modules, alcove paths and Macdonald polynomials},
		Selecta Mathematica, 2017,  Volume 23, Issue 4, pp. 2863--2897.
		
		
		\bibitem[FKM]{FKM}
		E.~Feigin, S.~Kato, I.~Makedonskyi, {\it Representation theoretic realization of non-symmetric Macdonald polynomials at infinity},
		J. Reine Angew. Math. 764 (2020), 181--216. 
		
		\bibitem[FKhM]{FKhM}
		{E. Feigin, A. Khoroshkin, and Ie.~Makedonskyi},
		{\it Duality theorems for current groups},
		Israel Journal of Mathematics, vol. 248, 441--479 (2022).
		
		\bibitem[FMO1]{FMO1}
		E. Feigin, I. Makedonskyi, D. Orr, {\it Generalized Weyl modules and nonsymmetric $q$-Whittaker functions},
		Adv. Math. 330 (2018), 997--1033.
		
		\bibitem[FMO2]{FMO2}
		E. Feigin, I. Makedonskyi, D. Orr, {\it Nonsymmetric q-Cauchy identity and
			representations of the Iwahori algebra},
		https://arxiv.org/abs/2303.00241.
		
		
		\bibitem[FKhMO]{FKhMO}
		E. Feigin, A. Khoroshkin, I. Makedonskyi, D. Orr, {\it Peter-Weyl theorem for Iwahori groups and highest weight categories},
		arXiv:2307.02124. 
		
		
		
		\bibitem[FL]{FL} G. Fourier, P. Littelmann,
		{\it Weyl modules, Demazure modules, KR-modules, crystals, fusion products and limit constructions},
		Advances in Mathematics 211 (2007), no. 2, 566--593.
		
		\bibitem[G]{G}
		B.G. Goodberry, \emph{Partially-symmetric Macdonald polynomials}, Ph.D. Thesis, Virginia Polytechnic Institute and State University,
		2022.
		
		
		
		
		
		
		
		\bibitem[H]{H}
		M. Haiman.
		{\it Cherednik algebras, Macdonald polynomials and combinatorics.} International Congress of Mathematicians. Vol. III, 843--872, Eur. Math. Soc., Z\"{u}rich, 2006.
		
		
		\bibitem[I]{I} {B.~Ion},
		{\em Nonsymmetric Macdonald polynomials and Demazure characters},
		Duke Mathematical Journal {116}:2 (2003), 299--318.
		
		
		\bibitem[vdK]{vdK}
		W. van der Kallen,
		{\em Longest weight vectors and excellent filtrations.}
		Math. Z. 201 (1989), no. 1, 19-31. 
		
		\bibitem[Kat]{Kat}
		S.~Kato, {\it Demazure character formula for semi-infinite flag manifolds}, Math. Ann. 371, 1769 -- 1801 (2018).
		
		\bibitem[Kac]{Kac}
		V.~Kac, Infinite dimensional Lie algebras, 3rd ed.,
		Cambridge University Press, Cambridge, 1990.
		
		
		\bibitem[Kh]{Kh}
		A.~Khoroshkin, {\em Highest weight categories and Macdonald polynomials}, arXiv:1312.7053.
		
		\bibitem[Kl]{Kl}
		A.~Kleshchev, {\it Affine highest weight categories and affine quasihereditary algebras}, Proc. Lond. Math.
		Soc. (3), 110(4):841--882, 2015.
		
		
		\bibitem[KKhM]{KKhM}
		S.~Kato, A.~Khoroshkin, I.Makedonskyi,  {\it 
			Categorification of DAHA and Macdonald polynomials}, arXiv:2103.10009.
		
		\bibitem[Kr]{Krause}
		H.~Krause, 
		{\it Highest weight categories and recollements.} 
		Annales de l'Institut Fourier, 67(6):2679--2701.
		
		\bibitem[Kum]{Kum}
		S. Kumar, Kac–Moody Groups, Their Flag Varieties and Representation Theory. Progress in Mathematics, vol. 204 (Birkh\"auser, Boston, 2002)
		
		
		
		\bibitem[KL]{KL}
		J.-H. Kwon and H. Lee,
		{\it Affine RSK correspondence and crystals of level zero extremal weight modules}, arXiv:2203.00406.
		
		
		
		\bibitem[Lap]{Lap}
		L.Lapointe, \emph{$m$-Symmetric functions, non-symmetric Macdonald polynomials and positivity conjectures}, arXiv:2206.05177. 
		
		\bibitem[Las]{Las} A.~Lascoux, {\it Double crystal graphs}, In Studies in Memory of Issai Schur. Birkh\"auser Boston, 2003.
		
		
		\bibitem[LNSSS]{LNSSS}
		C. Lenart, S. Naito, D. Sagaki, A. Schilling, and M. Shimozono, {\it A uniform model for Kirillov-Reshetikhin crystals III: nonsymmetric Macdonald polynomials at $t=0$ and Demazure characters}. Transform. Groups 22 (2017), no. 4, 1041--1079.
		
		
		
		
		
		\bibitem[M1]{M1}
		I. G. Macdonald, {\it Symmetric functions and Hall polynomials}, second ed., Oxford University Press, 1995.
		
		
		\bibitem[M2]{M2}
		I.~G.~Macdonald,
		{\it Affine Hecke algebras and orthogonal polynomials},
		Cambridge Tracts in Mathematics, 157. Cambridge University Press, Cambridge, 2003.
		
		\bibitem[MN]{MN}
		K. Mimachi and M. Noumi,
		{\it A reproducing kernel for nonsymmetric Macdonald polynomials}, 
		Duke Math. J. 91 (1998), no. 3, 621-634. 
		
		\bibitem[NS]{NS}
		S.~Naito, D.~Sagaki, {\em Level-zero van der Kallen modules and specialization of nonsymmetric Macdonald polynomials at $t=\infty$}, Transform. Groups 26 (2021), no. 3, 1077--1111. 
		
		\bibitem[NNS1]{NNS1}
		S.~Naito, F.~Nomoto, D.~Sagaki,
		{\em Specialization of nonsymmetric Macdonald polynomials at $t=\infty$ and Demazure submodules of level-zero extremal weight modules}, Trans. Amer. Math. Soc. 370 (2018), no. 4, 2739--2783.
		
		\bibitem[NNS2]{NNS2}
		S.~Naito, F.~Nomoto, D.~Sagaki,
		{\em Representation-theoretic interpretation of Cherednik-Orr's recursion formula for the specialization of nonsymmetric Macdonald polynomials at $t=\infty$}, Transform. Groups 24 (2019), no. 1, 155--191.
		
		
		
		
		
		
		\bibitem[OS]{OS}
		D. Orr, M. Shimozono, {\it Specializations of nonsymmetric Macdonald-Koornwinder polynomials},
		J. Algebraic Combin. 47 (2018), no. 1, 91--127. 
		
		
		\bibitem[PW]{PW} F. Peter, H. Weyl, {\it Die Vollst\"andigkeit der primitiven Darstellungen einer geschlossenen kontinuierlichen
			Gruppe}, Math. Ann., 97: 737--755, 1927.
		
		
		
		\bibitem[RY]{RY} {A. Ram, M. Yip},
		{\it A combinatorial formula for Macdonald polynomials}, Advances in Mathematics, vol. 226 (1), 2011, pp. 309--331.
		
		
		
		\bibitem[Sa]{Sa} {Y.~Sanderson},
		{\em On the Connection Between Macdonald Polynomials and Demazure Characters},
		J. of Algebraic Combinatorics, {11} (2000), 269--275.
		
		
		\bibitem[Sch]{Sch}
		P.~Schl\"osser, \emph{ Intermediate Macdonald polynomials and their vector versions},  
		arXiv:2310.17362. 
		
		
		\bibitem[TY]{TY}
		P. Tauvel, R. W. T. Yu, {\it Lie algebras and algebraic groups}, Springer, 2009.
		
		\bibitem[W]{Weibel}
		C.A. Weibel, {\it An introduction to homological algebra}, Cambridge university press, 1994.
		
	\end{thebibliography}
\end{document}